\documentclass{amsart} 
\usepackage{lmodern}
\usepackage{microtype}
\usepackage[english]{babel} 
\usepackage{hyperref} 
\usepackage{mathrsfs} 
\usepackage{color} 
\usepackage{subcaption} 
\usepackage{tikz} 
\usepackage{pgfplots}

\pgfplotsset{compat=1.17}
\usepgfplotslibrary{fillbetween}

\numberwithin{equation}{section} 
\numberwithin{figure}{section} 
\numberwithin{table}{section} 

\theoremstyle{plain} 
\newtheorem{theorem}{Theorem}[section] 
\newtheorem{lemma}[theorem]{Lemma} 
\newtheorem{corollary}[theorem]{Corollary}

\theoremstyle{remark} 
\newtheorem*{remark}{Remark}

\theoremstyle{definition} 
\newtheorem{conjecture}{Conjecture}[section] 
\newtheorem{question}[conjecture]{Question}

\DeclareMathOperator{\mre}{Re}

\begin{document} 
\title{Hilbert points in Hardy spaces} 
\date{\today} 

\author{Ole Fredrik Brevig} 
\address{Department of Mathematics, University of Oslo, 0851 Oslo, Norway} 
\email{obrevig@math.uio.no}

\author{Joaquim Ortega-Cerd\`{a}} 
\address{Department de Matem\`{a}tiques i Inform\`{a}tica, Universitat de Barcelona, Gran Via 585, 08007 Barcelona, Spain} 
\email{jortega@ub.edu}

\author{Kristian Seip} 
\address{Department of Mathematical Sciences, Norwegian University of Science and Technology (NTNU), NO-7491 Trondheim, Norway} 
\email{kristian.seip@ntnu.no}

\dedicatory{Dedicated, with admiration, to Nikolai Nikolski on the occasion of his 80th birthday.}

\begin{abstract}
	A Hilbert point in $H^p(\mathbb{T}^d)$, for $d\geq1$ and $1\leq p \leq \infty$, is a nontrivial function $\varphi$ in $H^p(\mathbb{T}^d)$ such that $\| \varphi \|_{H^p(\mathbb{T}^d)} \leq \|\varphi + f\|_{H^p(\mathbb{T}^d)}$ whenever $f$ is in $H^p(\mathbb{T}^d)$ and orthogonal to $\varphi$ in the usual $L^2$ sense. When $p\neq 2$, $\varphi$ is a Hilbert point in $H^p(\mathbb{T})$ if and only if $\varphi$ is a nonzero multiple of an inner function. An inner function on $\mathbb{T}^d$ is a Hilbert point in any of the spaces $H^p(\mathbb{T}^d)$, but there are other Hilbert points as well when $d\geq 2$. We investigate the case of $1$-homogeneous polynomials in depth and obtain as a byproduct a new proof of the sharp Khintchin inequality for Steinhaus variables in the range $2<p<\infty$. We also study briefly the dynamics of a certain nonlinear projection operator that characterizes Hilbert points as its fixed points. We exhibit an example of a function $\varphi$ that is a Hilbert point in $H^p(\mathbb{T}^3)$ for $p=2, 4$, but not for any other $p$; this is verified rigorously for $p>4$ but only numerically for $1\leq p<4$. 
\end{abstract}


\thanks{Ortega-Cerd\`{a} was partially supported by the Generalitat de Catalunya (grant 2017 SGR 358) and the Spanish Ministerio de Ciencia, Innovaci\'on y Universidades (project MTM2017-83499-P). Seip was supported in part by the Research Council of Norway grant 275113}

\maketitle

\section{Introduction} The prominence of inner functions (see for example \cite{Nikolski2} or \cite{Nikolski}) arose from Beurling's landmark paper \cite{Beurling} on the shift operator on the Hardy space $H^2(\mathbb{T})$. One usually defines an inner function on the unit disc $\mathbb{D}$ as a bounded analytic function whose nontangential limits are unimodular at almost every point of the unit circle $\mathbb{T}$. In the spirit of Beurling's theorem, one could alternatively define inner functions as the norm $1$ extremizers for point evaluation at the origin in invariant subspaces for the shift operator on $H^2(\mathbb{T})$, with an obvious modification should all functions in the space in question vanish at the origin. 

The point of departure of this paper is another extremal property of inner functions that characterizes them in the one-variable case but leads to a wider and intrinsically interesting class of functions on the $d$-dimensional torus $\mathbb{T}^d$ for $d>1$. The crucial definition is as follows. A nontrivial function $\varphi$ in $H^p(\mathbb{T}^d)$ for $1\leq p \leq \infty$ is said to be a \emph{Hilbert point} in $H^p(\mathbb{T}^d)$ if 
\begin{equation}\label{eq:cpoint} 
	\|\varphi\|_{H^p(\mathbb{T}^d)} \leq \|\varphi+f\|_{H^p(\mathbb{T}^d)} 
\end{equation}
for every $f$ in $H^p(\mathbb{T}^d)$ such that $\langle f, \varphi \rangle=0$, where $\langle \cdot, \cdot \rangle$ is the usual inner product in $L^2(\mathbb{T}^d)$. Here no precaution is needed when $p\geq 2$; when $1\leq p<2$, we declare that $\langle f, \varphi \rangle=0$ if $f$ lies in the closure in $H^p(\mathbb{T}^d)$ of the space of polynomials $g$ for which $\langle g,\varphi \rangle=0$. We will see from \eqref{eq:duality} below that, a posteriori, this precaution is obsolete because a Hilbert point in $H^p(\mathbb{T}^d)$ automatically belongs to the dual space $(H^p(\mathbb{T}^d))^\ast$. All nontrivial functions in $H^2(\mathbb{T}^d)$ are clearly Hilbert points in $H^2(\mathbb{T}^d)$. 

Our usage of the term ``Hilbert point'' is intended to suggest that we are dealing with points in a Banach space around which the space locally ``looks like'' a Hilbert space. This point of view is perhaps most succinctly reinforced by the following interpretation in terms of Banach space geometry. Given a fixed function $\varphi$, we use the notation 
\begin{equation}\label{eq:balls} 
	B_p: = \left\{f\in H^p(\mathbb{T}^d)\,:\, \|f\|_{H^p(\mathbb{T}^d)} \leq \|\varphi\|_{H^p(\mathbb{T}^d)}\right\} 
\end{equation}
on the presumption that $\varphi$ is in $H^p(\mathbb{T}^d)$. When $1<p<\infty$, we will see that a function $\varphi$ in $H^2(\mathbb{T}^d)\cap H^p(\mathbb{T}^d)$ is a Hilbert point in $H^p(\mathbb{T}^d)$ if and only if the supporting hyperplane $T_2$ to $B_2$ that contains the point $\varphi$, coincides with the supporting hyperplane $T_p$ to $B_p$ that contains $\varphi$, in the sense that $T_2 \cap H^p(\mathbb{T}^d) = T_p \cap H^2(\mathbb{T}^d)$. 

When $1\leq p < \infty$, we will investigate Hilbert points in $H^p(\mathbb{T}^d)$ using duality techniques. A consequence of the Hahn--Banach theorem is the following description. A nontrivial function $\varphi$ is a Hilbert point in $H^p(\mathbb{T}^d)$ for $1\leq p < \infty$ if and only if there is a constant $\lambda>0$ such that 
\begin{equation}\label{eq:duality} 
	P\left(|\varphi|^{p-2} \varphi\right) = \lambda \varphi, 
\end{equation}
where $P$ denotes the \emph{Riesz projection} from $L^2(\mathbb{T}^d)$ to $H^2(\mathbb{T}^d)$. The case $p=\infty$ is less amenable to duality arguments; we will see that it often requires separate arguments. 

Recall that $I$ in $H^p(\mathbb{T}^d)$ is said to be an \emph{inner function} if $|I(z)|=1$ for almost every $z$ in $\mathbb{T}^d$. From \eqref{eq:duality} it is evident that if $\varphi = C I$ for a constant $C\neq0$, then $\varphi$ is a Hilbert point in $H^p(\mathbb{T}^d)$ for every $p < \infty$. We obtain the same conclusion for the endpoint $p=\infty$ by taking the limit in \eqref{eq:cpoint}. Our first main result, alluded to above, asserts that there are no other Hilbert points in $H^p(\mathbb{T})$ when $p\neq 2$. 
\begin{theorem}\label{thm:inner1} 
	Fix $1 \leq p \leq \infty$, $p\neq 2$. A nontrivial function $\varphi$ is a Hilbert point in $H^p(\mathbb{T})$ if and only if $\varphi$ is a nonzero multiple of an inner function. 
\end{theorem}

The situation becomes rather more complicated when $d\geq 2$. In what follows, we will mainly restrict our attention to one of the simplest nontrivial subspaces of $H^p(\mathbb{T}^d)$, namely that of $1$-homogeneous polynomials. This means that we will be dealing with functions of the form 
\begin{equation}\label{eq:linear} 
	\varphi(z) = \sum_{j=1}^d c_j z_j. 
\end{equation}
Theorem~\ref{thm:linear1} below reveals that there are Hilbert points in this subspace with function theoretic properties that effectively contrast those of inner functions.

Our study of $1$-homogeneous polynomials as Hilbert points is chiefly based on \eqref{eq:duality} and the following remarkable formula. 
\begin{theorem}\label{thm:Plinear} 
	Fix $1 \leq p < \infty$ and suppose that $\varphi(z)=\sum_{j=1}^d c_j z_j$. Then
	\[\left(P|\varphi|^{p-2} \varphi\right)(z) = \frac{p}{2}\sum_{j=1}^d c_j z_j \int_0^1 \int_{\mathbb{T}^d} |\varphi_j(\zeta,r)|^{p-2}\,dm_d(\zeta)\,2rdr\]
	where $\varphi_j(z,r) := \varphi(z_1,\ldots,r z_j,\ldots,z_d)$ for $j=1,2,\ldots,d$. 
\end{theorem}

The integrals on the right-hand side of the formula in Theorem~\ref{thm:Plinear} only depend on the modulus of the coefficients of the $1$-homogeneous polynomial \eqref{eq:linear}. By symmetry, we can therefore easily obtain the following result from \eqref{eq:duality} and Theorem~\ref{thm:Plinear} for $p<\infty$, and then for $p=\infty$ using \eqref{eq:cpoint}. 
\begin{theorem}\label{thm:linear1} 
	If the nonzero coefficients of $\varphi(z)=\sum_{j=1}^d c_j z_j$ all have the same modulus, then $\varphi$ is a Hilbert point in $H^p(\mathbb{T}^d)$ for every $1\leq p\leq \infty$. 
\end{theorem}
Functions of the form $\varphi(z)=\sum_{j=1}^d c_j z_j$ whose coefficients all have the same positive modulus, maximize the ratio $\| \varphi \|_{H^{\infty}(\mathbb{T}^d)}/\| \varphi \|_{H^2(\mathbb{T}^d)}$ among all $1$-homogeneous polynomials, in stark contrast to what inner functions do. We are interested in whether there is any other possible choice of coefficients in \eqref{eq:linear} that yields Hilbert points for some $p\neq2$. We will obtain the following partial converse to Theorem~\ref{thm:linear1}. 
\begin{theorem}\label{thm:linear2} 
	Suppose that $2<p\leq\infty$. If $\varphi(z)=\sum_{j=1}^d c_j z_j$ is a Hilbert point in $H^p(\mathbb{T}^d)$, then the nonzero coefficients of $\varphi$ all have the same modulus. 
\end{theorem}

We conjecture that Theorem~\ref{thm:linear2} is true also for $1 \leq p < 2$. To obtain some evidence supporting this conjecture, we consider the following dynamical system. Let $\varphi_0$ be any $1$-homogeneous polynomial with $\|\varphi_0\|_{H^2(\mathbb{T}^d)}=1$. Based on \eqref{eq:duality} and Theorem~\ref{thm:Plinear}, we iteratively define 
\begin{equation}\label{eq:iterates} 
	\varphi_{n+1}:=\frac{P\left(|\varphi_n|^{p-2} \varphi_n\right)}{\left\|P\left(|\varphi_n|^{p-2} \varphi_n\right)\right\|_{H^2(\mathbb{T}^d)}}. 
\end{equation}
In the range $2<p<\infty$, we can completely describe the behavior of this dynamical system. It turns out that given any $\varphi_0(z)=\sum_{j=1}^d c_j z_j$, the iterates \eqref{eq:iterates} will converge to a Hilbert point $\varphi(z)=\sum_{j=1}^d \widetilde{c_j} z_j$ with $\widetilde{c_j} \neq 0$ if and only if $c_j\neq 0$. 

Based on an analysis of the simple case $d=2$ and numerical experiments in the case $d=3$, we observe that when $1 \leq p < 2$, the iterates will also converge to a Hilbert point, but now to a $1$-homogeneous inner function, i.e., to a unimodular multiple of $z_j$ for some $j$. If this convergence could be established for $d\geq3$, we would have a proof of Theorem~\ref{thm:linear2} also in the range $1 \leq p < 2$. 

There is an interesting connection between $1$-homogeneous polynomials that are Hilbert points in $H^p(\mathbb{T}^d)$ and the sharp Khintchin inequality for Steinhaus variables. To see this, we begin by setting
\[a_p := \min\left(1,\,{\Gamma\left(1+\frac{p}{2}\right)}^\frac{1}{p}\right) \qquad \text{and} \qquad b_p := \max\left(1,\,{\Gamma\left(1+\frac{p}{2}\right)}^\frac{1}{p}\right)\]
for $1 \leq p < \infty$. Khintchin's inequality for Steinhaus variables can be formulated as the estimates 
\begin{equation}\label{eq:Khintchin} 
	a_p \|\varphi\|_{H^2(\mathbb{T}^d)} \leq \|\varphi\|_{H^p(\mathbb{T}^d)} \leq b_p \|\varphi\|_{H^2(\mathbb{T}^d)} 
\end{equation}
for $1$-homogeneous polynomials $\varphi$. The upper estimate when $1 \leq p<2$ and the lower estimate when $2<p\leq \infty$ are trivial consequences of H\"older's inequality. Otherwise, the constants $a_p$ and $b_p$ are optimal as $d\to\infty$. The case $p=1$ was established by Sawa \cite{Sawa85}, and the case $1<p<2$ was proved by Kwapie\'{n} and K\"{o}nig~\cite{KK01}. The final case $2<p<\infty$ is independently due to Baernstein and Culverhouse \cite{BC02} and to Kwapie\'{n} and K\"{o}nig \cite{KK01}. The best constant in \eqref{eq:Khintchin} is also known in the case $0<p<1$ by a result of K\"onig \cite{Konig2014}.

The connection between Hilbert points in $H^p(\mathbb{T}^d)$ and Khintchin's inequality is as follows. 
\begin{lemma}\label{lem:Khintchin} 
	Fix $d\geq1$ and $1\leq p < \infty$. Consider the functional defined on the unit sphere of $\mathbb{C}^d$ by
	\[\mathscr{K}_p(c) := \|c_1 z_1 + \cdots + c_d z_d\|_{H^p(\mathbb{T}^d)}.\]
	Then $c=(c_1,\ldots,c_d)$ is a critical point of $\mathscr{K}_p$ if and only if $\varphi(z) = c_1 z_1 + \cdots + c_d z_d$ is a Hilbert point in $H^p(\mathbb{T}^d)$. 
\end{lemma}

Combining Theorem~\ref{thm:linear1}, Theorem~\ref{thm:linear2}, and Lemma~\ref{lem:Khintchin} with a computation, we obtain a new proof of the sharp Khintchin inequality for Steinhaus variables in the case $2<p<\infty$. In our proof, the heavy lifting is all done by Theorem~\ref{thm:Plinear}. 

In view of the results presented above, one might be tempted to conjecture that if $\varphi$ is a Hilbert point in $H^p(\mathbb{T}^d)$ for \emph{some} $p\neq2$, then $\varphi$ is a Hilbert point in $H^p(\mathbb{T}^d)$ for \emph{all} $1 \leq p \leq \infty$. We will however show that this is not the true. Specifically, we will consider 
\begin{equation}\label{eq:3hom} 
	\varphi(z) = z_1^3+z_2^3+z_1z_2z_3. 
\end{equation}
Using \eqref{eq:duality} and an argument involving change of variables, we will see that $\varphi$ is a Hilbert point in $H^p(\mathbb{T}^3)$ if and only if a certain Fourier coefficient of the function $|\zeta_1+\zeta_2+\zeta_3|^{p-2}(\zeta_1+\zeta_2+\zeta_3)$ vanishes. This allows us to establish that \eqref{eq:3hom} is a Hilbert point in $H^2(\mathbb{T}^3)$ and $H^4(\mathbb{T}^3)$, but not in $H^p(\mathbb{T}^3)$ for $4<p\leq\infty$. Based on numerical evidence, we conjecture that $\varphi$ is neither a Hilbert point in $H^p(\mathbb{T}^3)$ for $1 \leq p < 2$ nor for $2<p<4$. 

To close this introduction, we give a brief overview of the contents of this paper. In Section~\ref{sec:prelim}, we reformulate our problem using duality techniques and establish Theorem~\ref{thm:inner1}. Section~\ref{sec:linear} is devoted to the proof of Theorem~\ref{thm:Plinear}, Theorem~\ref{thm:linear1} and Theorem~\ref{thm:linear2}. The dynamical system mentioned above is investigated in Section~\ref{sec:dynamics}. In Section~\ref{sec:Khintchin}, we prove Lemma~\ref{lem:Khintchin} and present a new proof of Khintchin's inequality for Steinhaus variables in the range $2<p<\infty$. Finally, the function \eqref{eq:3hom} is discussed in detail in Section~\ref{sec:example}.

\section{Duality reformulation and inner functions} \label{sec:prelim}

Recall that every function $f$ in $L^p(\mathbb{T}^d)$ can be represented by its Fourier series $f(z) \sim \sum_{\alpha \in \mathbb{Z}^d} \widehat{f}(\alpha)\,z^\alpha$, where
\[\widehat{f}(\alpha) := \int_{\mathbb{T}^d} f(z) \,\overline{z^\alpha}\,dm_d(z)\]
and where $m_d$ denotes the Haar measure of the $d$-dimensional torus $\mathbb{T}^d$. The Hardy space $H^p(\mathbb{T}^d)$ is the subspace of $L^p(\mathbb{T}^d)$ comprised of functions $f$ such that $\widehat{f}(\alpha)=0$ for every $\alpha$ in $\mathbb{Z}^d \setminus \mathbb{N}_0^d$, where $\mathbb{N}_0 := \{0,1,2,\ldots\}$.

We need a well-known consequence of the Hahn--Banach theorem concerning orthogonality in $L^p$ spaces, which can be extracted from Shapiro's monograph \cite{Shapiro}. 
\begin{lemma}\label{lem:shapiro} 
	Fix $1 \leq p < \infty$. If $\varphi$ is a nontrivial function in $H^p(\mathbb{T}^d)$ and $Y$ is a closed subspace of $L^p(\mathbb{T}^d)$, then the following are equivalent. 
	\begin{enumerate}
		\item[(i)] $\|\varphi\|_{L^p(\mathbb{T}^d)} \leq \|\varphi+f\|_{L^p(\mathbb{T}^d)}$ for every $f$ in $Y$. 
		\item[(ii)] $\left\langle |\varphi|^{p-2} \varphi, f \right\rangle =0$ for every $f$ in $Y$. 
	\end{enumerate}
\end{lemma}
\begin{proof}
	If $1<p<\infty$, this is a special case of \cite[Thm.~4.2.1]{Shapiro}. If $\varphi$ is a nontrivial function in $H^1(\mathbb{T}^d)$, then $\log|\varphi|$ is in $L^1(\mathbb{T}^d)$ by \cite[Thm.~3.3.5]{Rudin69}. In particular,
	\[m_d\left(\{z \in \mathbb{T}^d\,:\, \varphi(z)=0\}\right)=0.\]
	This means that we obtain the assertion for $p=1$ from \cite[Thm.~4.2.2]{Shapiro}. 
\end{proof}

If $\varphi$ is in $(H^p(\mathbb{T}^d))^\ast$, then the bounded linear functional generated by $\varphi$ on $H^p(\mathbb{T}^d)$ can be represented as $L_\varphi(f):= \langle f, \varphi \rangle$, which implies that 
\begin{equation}\label{eq:dualnorm} 
	\|\varphi\|_{(H^p(\mathbb{T}^d))^\ast} := \sup_{\substack{g \in H^p(\mathbb{T}^d) \\
	g \not \equiv 0}} \frac{|\langle g, \varphi \rangle|}{\|g\|_{H^p(\mathbb{T}^d)}}. 
\end{equation}
We are now ready to reformulate the defining property of Hilbert points using duality. Note that the condition in Theorem~\ref{thm:duality} (a) below coincides with \eqref{eq:duality} discussed in the introduction. 
\begin{theorem}\label{thm:duality} 
	\mbox{} 
	\begin{enumerate}
		\item[(a)] Fix $1 \leq p < \infty$. A nontrivial function $\varphi$ in $H^p(\mathbb{T}^d)$ is a Hilbert point in $H^p(\mathbb{T}^d)$ if and only if there is some constant $\lambda>0$ such that
		\[P\left(|\varphi|^{p-2} \varphi\right) = \lambda \varphi.\]
		\item[(b)] Fix $1 \leq p \leq \infty$. A nontrivial function $\varphi$ in $H^p(\mathbb{T}^d)\cap H^2(\mathbb{T}^d)$ is a Hilbert point in $H^p(\mathbb{T}^d)$ if and only if
		\[\|\varphi\|_{H^p(\mathbb{T}^d)} \|\varphi\|_{(H^p(\mathbb{T}^d))^\ast} = \|\varphi\|_{H^2(\mathbb{T}^d)}^2.\]
	\end{enumerate}
\end{theorem}

Part (a) implies that a Hilbert point in $H^p(\mathbb{T}^d)$ also belongs to the dual space $(H^p(\mathbb{T}^d))^\ast$. This means that the condition of part (b) that $\varphi$ be in $H^2(\mathbb{T}^d)$ is automatically verified when $\varphi$ is a Hilbert point in $H^p(\mathbb{T}^d)$. The assumption that $\varphi$ be in $H^2(\mathbb{T}^d)$ is only needed to exclude from the statement the claim that if $1\leq p < 2$, then any $\varphi$ in $H^p(\mathbb{T}^d) \setminus H^2(\mathbb{T}^d)$ is a Hilbert point in $H^p(\mathbb{T}^d)$.

The formula in part (b) further justifies our usage of the term ``Hilbert point'', since the identity $\| \varphi \|_{\mathscr{H}} \| \varphi \|_{\mathscr{H}^\ast} =\| \varphi\|_{\mathscr{H}}^2 $ holds for all vectors $\varphi$ in a Hilbert space $\mathscr{H}$ by the Riesz representation theorem.
\begin{remark}
	Theorem~\ref{thm:linear1} can also be deduced from \cite[Lem.~5]{Brevig19} and Theorem~\ref{thm:duality} (b), in addition to the proof based on Theorem~\ref{thm:Plinear} and Theorem~\ref{thm:duality} (a) mentioned above. 
\end{remark}
\begin{proof}
	[Proof of Theorem~\ref{thm:duality} (a)] Throughout the proof, we will apply Lemma~\ref{lem:shapiro} with $Y$ as the closure in $L^p(\mathbb{T}^d)$ of the set of analytic polynomials $g$ satisfying $\langle g, \varphi\rangle =0$.
	
	We assume first that $P\left(|\varphi|^{p-2} \varphi\right) = \lambda \varphi$ for some $\lambda>0$. By the implication (ii) $\implies$ (i) in Lemma~\ref{lem:shapiro}, this means that
	\[\|\varphi\|_{H^p(\mathbb{T}^d)}\leq \| \varphi+f\|_{H^p(\mathbb{T}^d)} \]
	for every $f$ in $Y$. This shows that $\varphi$ is a Hilbert point in $H^p(\mathbb{T}^d)$.
	
	To prove the reverse implication, we begin by assuming that $\varphi$ is a Hilbert point in $H^p(\mathbb{T}^d)$. By the implication (i) $\implies$ (ii) in Lemma~\ref{lem:shapiro}, we have that
	\[\left\langle |\varphi|^{p-2} \varphi, f \right\rangle = 0 \]
	for every $f$ in $Y$. But this means that the function $\psi:=P\left(|\varphi|^{p-2} \varphi\right)$ also has the property that $\langle f,\psi \rangle =0$ for all $f$ in $Y$. Since $\varphi$ is in $H^p(\mathbb{T}^d)$ and $\psi$ is in $(H^p(\mathbb{T}^d))^\ast$, at least one of them belongs to $H^2(\mathbb{T}^d)$. If $1\leq p<2$ so that $\psi$ is in $H^2(\mathbb{T}^d)$, then every $f$ in $H^p(\mathbb{T}^d)$ can be decomposed as
	\[f=\frac{\langle f, \psi \rangle}{\| \psi\|_{H^2(\mathbb{T}^d)}^2} \psi + h, \]
	where $h$ belongs to $Y$. Since $H^{q}(\mathbb{T}^d)$ is contained in $H^p(\mathbb{T}^d)$ for $q=p/(p-1)$, this decomposition is in particular valid for every $f$ in $H^q(\mathbb{T}^d)$. It follows that the action of the functional $L_{\varphi}$ on $H^q(\mathbb{T}^d)$ can be computed explicitly:
	\[L_{\varphi}(f)=\frac{\langle f, \psi \rangle}{\| \psi\|_{H^2(\mathbb{T}^d)}^2} \langle \psi, \varphi \rangle.\]
	Since $\langle \psi, \varphi \rangle = \langle |\varphi|^{p-2} \varphi, \varphi\rangle = \|\varphi\|_{H^p(\mathbb{T}^d)}^p$, this means that $\varphi$ must be a positive multiple of $\psi$. When $2<p<\infty$, we may argue in the same way, with the roles of $\varphi$ and $\psi$ reversed. 
\end{proof}
\begin{proof}
	[Proof of Theorem~\ref{thm:duality} (b)] If $\varphi$ is a Hilbert point in $H^p(\mathbb{T}^d)$, then $\varphi$ is also in the dual space $(H^p(\mathbb{T}^d))^\ast$. When $2 \leq p \leq \infty$, this is trivial. When $1 \leq p < 2$, this follows from Theorem~\ref{thm:duality} (a) and the fact that $|\varphi|^{p-2}\varphi$ is in $L^q(\mathbb{T}^d)$, where $q=p/(p-1)$. The inequality
	\[ \|\varphi\|_{H^p(\mathbb{T}^d)} \|\varphi\|_{(H^p(\mathbb{T}^d))^\ast} \geq \|\varphi\|_{H^2(\mathbb{T}^2)}^2\]
	holds automatically by \eqref{eq:dualnorm}, so it suffices to prove that $\varphi$ is a Hilbert point in $H^p(\mathbb{T}^d)$ if and only if the reverse inequality 
	\begin{equation}\label{eq:reverse} 
		\|\varphi\|_{H^p(\mathbb{T}^d)} \|\varphi\|_{(H^p(\mathbb{T}^d))^\ast} \leq \|\varphi\|_{H^2(\mathbb{T}^2)}^2 
	\end{equation}
	is verified.
	
	We begin with the necessity of \eqref{eq:reverse}. To this end, we assume that $\varphi$ is a Hilbert point in $H^p(\mathbb{T}^d)$. Since $\varphi$ is in $(H^p(\mathbb{T}^d))^\ast$ and since $H^p(\mathbb{T}^d) \cap (H^p(\mathbb{T}^d))^\ast\subset H^2(\mathbb{T}^d)$, we may decompose every $g$ in $H^p(\mathbb{T}^d)$ as 
	\begin{equation}\label{eq:orthodecomp} 
		g = \frac{\langle g, \varphi \rangle}{\|\varphi\|_{H^2(\mathbb{T}^d)}^2} \varphi + \left(g- \frac{\langle g, \varphi \rangle}{\|\varphi\|_{H^2(\mathbb{T}^d)}^2} \varphi\right). 
	\end{equation}
	If $g\not\equiv 0$, then we use the decomposition \eqref{eq:orthodecomp} and the assumption that $\varphi$ is a Hilbert point in $H^p(\mathbb{T})$ to see that
	\[\frac{|\langle g, \varphi \rangle|}{\|g\|_{H^p(\mathbb{T}^d)}} \leq \frac{\|\varphi\|_{H^2(\mathbb{T}^d)}^2}{\|\varphi\|_{H^p(\mathbb{T}^d)}}.\]
	Since $g$ is arbitrary, we get the desired inequality \eqref{eq:reverse}.
	
	To prove the sufficiency of \eqref{eq:reverse}, we suppose next that $\varphi$ satisfies \eqref{eq:reverse}. By \eqref{eq:dualnorm} we then get that
	\[\frac{|\langle g, \varphi \rangle|}{\|g\|_{H^p(\mathbb{T}^d)}} \leq \frac{\|\varphi\|_{H^2(\mathbb{T}^d)}}{\|\varphi\|_{H^p(\mathbb{T}^d)}}\]
	for every nontrivial $g$ in $H^p(\mathbb{T}^d)$. In particular, choosing $g = \varphi + f$ with $\langle f, \varphi \rangle=0$, we see that $\varphi$ is a Hilbert point in $H^p(\mathbb{T}^d)$. 
\end{proof}

We will now formulate three corollaries of Theorem~\ref{thm:duality}. The first just makes explicit an immediate consequence of the fact that a Hilbert point $\varphi$ in any of the spaces $H^p(\mathbb{T}^d)$ is in $H^2(\mathbb{T}^d)\cap (H^p(\mathbb{T}^d))^\ast$. Then the orthogonal projection
\[ P_\varphi f :=\frac{\langle f,\varphi \rangle}{\|\varphi\|_{H^2(\mathbb{T}^d)}} \varphi \]
is a well defined operator on $H^p(\mathbb{T}^d)$, and we obtain the following from part (b). 
\begin{corollary}\label{cor:oproj} 
	A nontrivial function $\varphi$ in $H^p(\mathbb{T}^d)$ is a Hilbert point in $H^p(\mathbb{T}^d)$ for $1\leq p \leq \infty$ if and only if $\varphi$ belongs to $H^2(\mathbb{T}^d)\cap (H^p(\mathbb{T}^d))^\ast$ and $P_\varphi$ is a contraction on $H^p(\mathbb{T}^d)$. 
\end{corollary}
The Hahn--Banach theorem supplies a contractive projection onto every one-dimensional subspace of a Banach space. Corollary~\ref{cor:oproj} can be reformulated as follows. If $\varphi$ is in $H^2(\mathbb{T}^d)\cap (H^p(\mathbb{T}^d))^\ast$, then $\varphi$ is a Hilbert point in $H^p(\mathbb{T}^d)$ if and only if the projection from $H^p(\mathbb{T}^d)$ to $\operatorname{span}(\{\varphi\})$ coincides with the orthogonal projection from $H^2(\mathbb{T}^d)$ to $\operatorname{span}(\{\varphi\})$. We refer to \cite{LRR05} and our recent paper \cite{BOS21} for studies of other contractive projections on Hardy spaces.

We next come back to our interpretation of Hilbert points in terms of Banach space geometry. We retain the notation from the introduction (see \eqref{eq:balls}) and stress that the supporting hyperplane $T_p$ of $B_p$ is well defined since $H^p(\mathbb{T}^d)$ is uniformly convex when $1<p<\infty$.
\begin{corollary}\label{cor:hyperplane} 
	Suppose that $\varphi$ is in $H^2(\mathbb{T}^d)\cap H^p(\mathbb{T}^d)$ with $1<p<\infty$. Then $\varphi$ is a Hilbert point in $H^p(\mathbb{T}^d)$ if and only if $T_p\cap H^2(\mathbb{T}^d) = T_2\cap H^p(\mathbb{T}^d)$. 
\end{corollary}
\begin{proof}
	We begin by assuming that $\varphi$ is a Hilbert point in $H^p(\mathbb{T}^d)$. Then the inclusion $T_2\cap H^p(\mathbb{T}^d) \subset T_p\cap H^2(\mathbb{T}^d)$ is immediate. To prove the reverse inclusion, we begin by noting that if $f$ is in $T_p$, then
	\[\|\varphi + f \|_{H^p(\mathbb{T}^d)} \geq \|\varphi\|_{H^p(\mathbb{T}^d)}.\]
	Thus Lemma~\ref{lem:shapiro} implies that $\langle f, |\varphi|^{p-2}\varphi \rangle = 0$ for every $f$ in $T_p$. Therefore, since $P$ is self-adjoint, we have
	\[ \left\langle f, P(|\varphi|^{p-2}\varphi) \right\rangle = 0. \]
	Now invoking Theorem~\ref{thm:duality} (a), we see that $\langle f, \varphi \rangle = 0$, whence $f$ is in $T_2$.
	
	We assume next that $T_p\cap H^2(\mathbb{T}^d) = T_2\cap H^p(\mathbb{T}^d)$. If $\langle f,\varphi \rangle =0$, then $f$ will be in $T_p$ which implies that $\|\varphi+f\|_{H^p(\mathbb{T}^d)} \geq \|\varphi\|_{H^p(\mathbb{T}^d)}$. This means by definition that $\varphi$ is a Hilbert point in $H^p(\mathbb{T}^d)$. 
\end{proof}

As mentioned in the introduction, the following result is a direct consequence of Theorem~\ref{thm:duality} (a) for $p<\infty$ and by a limiting argument for $p=\infty$. It is also possible to deduce this result from Theorem~\ref{thm:duality} (b), \eqref{eq:dualnorm}, and H\"older's inequality. 
\begin{corollary}\label{cor:inneryes} 
	Fix $d\geq1$ and suppose that $\varphi= CI$ for a constant $C\neq0$ and an inner function $I$. Then $\varphi$ is a Hilbert point in $H^p(\mathbb{T}^d)$ for every $1 \leq p \leq \infty$. 
\end{corollary}

We now turn to the proof of Theorem~\ref{thm:inner1}, which states that there are no other Hilbert points in $H^p(\mathbb{T})$ when $p\neq 2$. 
\begin{proof}
	[Proof of Theorem~\ref{thm:inner1}] The sufficiency part is the case $d=1$ of Corollary~\ref{cor:inneryes}, so it remains to settle the necessity part. 
	
	We begin with the case $1 \leq p < \infty$, $p\neq 2$. We assume that $\varphi$ is a Hilbert point in $H^p(\mathbb{T}^d)$ and use Theorem~\ref{thm:duality} (a) to infer that
	\[P\left(|\varphi|^{p-2}\varphi\right) = \lambda \varphi\]
	for some $\lambda>0$. We may assume without loss of generality that $\lambda=1$ by rescaling $\varphi$ if necessary. Suppose that $f$ is an arbitrary analytic polynomial. Then since the Riesz projection is self-adjoint, we find that
	\[\left\langle |\varphi|^2, f \right\rangle = \left\langle \varphi, \varphi f \right\rangle = \left\langle P (|\varphi|^{p-2} \varphi), \varphi f \right\rangle = \left\langle |\varphi|^{p-2} \varphi, \varphi f \right\rangle = \left\langle |\varphi|^p, f \right\rangle.\]
	The same identity holds also with $f$ replaced by $\overline{f}$. Hence $|\varphi|^p = |\varphi|^2$ almost everywhere, so $\varphi$ must be an inner function.
	
	We assume next that $\varphi$ is a Hilbert point in $H^\infty(\mathbb T)$. We factor $\varphi$ in the usual way as $\varphi = I E$, where $I$ is an inner function and $E$ is an outer function. For a real number $\eta$, consider the function $f_\eta = I E^{1+\eta}$ which is in $H^\infty(\mathbb{T})$ and satisfies $\|f_\eta\|_{H^\infty(\mathbb{T})} = \|\varphi\|_{H^\infty(\mathbb{T})}^{1+\eta}$. Since 
	\begin{equation}\label{eq:dunorm} 
		\| \varphi \|_{(H^{\infty}(\mathbb{T}))^*} \geq \mre \frac{\langle f_\eta, \varphi\rangle}{\|f_\eta\|_{H^\infty(\mathbb{T})}}, 
	\end{equation}
	Theorem~\ref{thm:duality} (b) shows that $\varphi$ is a Hilbert point in $H^\infty(\mathbb{T})$ only if the quantity to the right in \eqref{eq:dunorm} is maximized for $\eta=0$. Using that $|\varphi|=|E|$ almost everywhere, we find that
	\[0=\left.\frac{d}{d\eta} \mre \frac{\langle f_\eta, \varphi\rangle}{\|f_\eta\|_{H^\infty(\mathbb{T})}} \right|_{\eta=0} = \frac{1}{\|E\|_{H^\infty}} \int_{\mathbb{T}} |E(z)|^2 \log\left(\frac{|E(z)|}{\|E\|_{H^\infty(\mathbb{T})}}\right)\,dm_1(z),\]
	which holds if and only if $\varphi$ is a constant multiple of an inner function. 
\end{proof}

\section{\texorpdfstring{$1$}{1}-homogeneous polynomials} \label{sec:linear} 
To prove Theorem~\ref{thm:Plinear} we require some basic facts. We first recall that the Riesz projection $P$ can be expressed using the Szeg\H{o} kernel as 
\begin{equation}\label{eq:szego} 
	Pf(z) = \int_{\mathbb{T}^d} f(\zeta) \prod_{j=1}^d \frac{1}{1-z_j \overline{\zeta_j}}\,dm_d(\zeta). 
\end{equation}
Next, a function $f$ in $L^p(\mathbb{T}^d)$ is called $1$-homogeneous if
\[f(e^{i\theta}z_1, e^{i\theta}z_2,\ldots, e^{i\theta}z_d) = e^{i\theta} f(z_1,z_2,\ldots,z_d).\]
It is clear that if $\varphi$ is a $1$-homogeneous polynomial, then $|\varphi|^{p-2} \varphi$ is a $1$-homogeneous function. Hence $P(|\varphi|^{p-2} \varphi)$ is a $1$-homogeneous polynomial whenever $\varphi$ is a $1$-homogeneous polynomial. 
\begin{proof}
	[Proof of Theorem~\ref{thm:Plinear}] Our goal is to establish that if $\varphi(z)=\sum_{j=1}^d c_j z_j$, then 
	\begin{equation}\label{eq:Plinear} 
		\left(P|\varphi|^{p-2} \varphi\right)(z) = \frac{p}{2}\sum_{j=1}^d c_j z_j \int_0^1 \int_{\mathbb{T}^d} |\varphi_j(\zeta,r)|^{p-2}\,dm_d(\zeta)\,2rdr, 
	\end{equation}
	where $\varphi_j(z,r) := \varphi(z_1,\ldots,r z_j,\ldots,z_d)$ for $j=1,2,\ldots,d$. By the discussion above, we know that the left-hand side of \eqref{eq:Plinear} is a $1$-homogeneous polynomial. Hence we do not need to consider any other Fourier coefficients when computing the Riesz projection.
	
	Let us first demonstrate that we without loss of generality may assume that $c_j>0$ for $j=1,2,\ldots,d$. Suppose that we have established \eqref{eq:Plinear} for $c_j>0$. Given any $\varphi$, we define
	\[\widetilde{\varphi}(z) = \sum_{j=1}^d e^{i\theta_j} c_j z_j\]
	where $e^{i\theta_j}$ is chosen so that $\widetilde{c_j} := e^{i\theta_j} c_j>0$. Using \eqref{eq:szego}, a change of variables and the rotational invariance of $\mathbb{T}^d$, we find that
	\[\left(P|\varphi|^{p-2} \varphi\right)(z) = \left(P|\widetilde{\varphi}|^{p-2} \widetilde{\varphi} \right)(e^{-i\theta_1} z_1, e^{-i\theta_2}z_2,\ldots,e^{-i\theta_d}z_d).\]
	Using \eqref{eq:Plinear} for $\widetilde{\varphi}$ we obtain \eqref{eq:Plinear} for $\varphi$, since the integrals on the right-hand side of \eqref{eq:Plinear} are the same for $\varphi$ and $\widetilde{\varphi}$, again by rotational invariance.
	
	Considering the Fourier series of $|\varphi|^{p-2}$, we compute 
	\begin{equation}\label{eq:Fproj} 
		(P|\varphi|^{p-2} \varphi)(z) = A \varphi(z) + \sum_{j=1}^d z_j \sum_{\substack{k=1 \\
		k \neq j}}^d B_{j,k} c_k , 
	\end{equation}
	for $A = \int_{\mathbb{T}^d} |\varphi(\zeta)|^{p-2}\,dm_d(\zeta)$ and $B_{j,k} = \int_{\mathbb{T}^d} |\varphi(\zeta)|^{p-2} \, \zeta_j \overline{\zeta_k}\,dm_d(\zeta)$ with $j\neq k$. In what follows, let $\alpha$ in $\mathbb{N}_0^d$ denote a multi-index, and set $c = (c_1,c_2,\ldots,c_d)$. Suppose that $p-2=2n$ for a nonnegative integer $n$. By the multinomial theorem, 
	\begin{equation}\label{eq:multithm} 
		(\varphi(z))^n = \sum_{|\alpha|=n} \binom{n}{\alpha} c^\alpha z^\alpha \qquad\text{where}\qquad \binom{n}{\alpha}=\frac{n!}{\alpha_1! \alpha_2! \cdots \alpha_d!}. 
	\end{equation}
	We will use \eqref{eq:multithm} to obtain expressions for $A$ and $B_{j,k}$. It is clear that
	\[A = \sum_{|\alpha|=n} \binom{n}{\alpha}^2 c^{2\alpha}.\]
	Given some $\alpha$ with $\alpha_k\geq1$ let $\beta$ denote the multi-index obtained by subtracting $1$ from the $k$th coordinate of $\alpha$ and adding $1$ to the $j$th coordinate of $\alpha$. Note that
	\[\binom{n}{\beta} = \binom{n}{\alpha} \frac{\alpha_k}{\alpha_j+1}.\]
	By using \eqref{eq:multithm} twice, we find that
	\[B_{j,k} = \sum_{\substack{|\alpha|=n \\
	\alpha_k \geq1}} \binom{n}{\alpha} \binom{n}{\beta} c^{\alpha} c^{\beta} = \sum_{\substack{|\alpha|=n \\
	\alpha_k \geq1}} \binom{n}{\alpha}^2 \frac{\alpha_k}{\alpha_j+1} \frac{c_j}{c_k} c^{2\alpha} = \frac{c_j}{c_k} \sum_{|\alpha|=n} \binom{n}{\alpha}^2 \frac{\alpha_k}{\alpha_j+1} c^{2\alpha}.\]
	Consequently, 
	\begin{equation}\label{eq:rowsum} 
		Ac_j+ \sum_{\substack{k=1 \\
		k \neq j}}^d B_{j,k} c_k = c_j (n+1) \sum_{|\alpha|=n} \binom{n}{\alpha}^2 \frac{c^{2\alpha}}{\alpha_j+1}. 
	\end{equation}
	A direct computation shows that
	\[\sum_{|\alpha|=n} \binom{n}{\alpha}^2 \frac{c^{2\alpha}}{\alpha_j+1} = \int_0^1 \int_{\mathbb{T}^d} |\varphi_j(\zeta,r)|^{2n}\,dm_d(\zeta)\,2rdr\]
	for $\varphi_j(z,r) = \varphi(z_1,\ldots,r z_j,\ldots,z_d)$ for $j=1,2,\ldots,d$. By \eqref{eq:Fproj} and \eqref{eq:rowsum} we have now established the formula 
	\begin{equation}\label{eq:2nproj} 
		(P|\varphi|^{2n} \varphi)(z) = (n+1) \sum_{j=1}^d c_j z_j \int_0^1 \int_{\mathbb{T}^d} |\varphi_j(\zeta,r)|^{2n}\,dm_d(\zeta)\,2rdr 
	\end{equation}
	for $n=0,1,2,\ldots$. We can extend \eqref{eq:2nproj} to the case $n=p-2$ for $p\geq1$ by polynomial approximation. Let $D$ denote the differential operator
	\[Df(x) := \frac{d}{dx} \big(x f(x)\big).\]
	If $f(x) = x^n$ for $n=0,1,2,\ldots$, then \eqref{eq:2nproj} can be restated as 
	\begin{equation}\label{eq:fproj} 
		P\left(f\big(|\varphi|^2\big)\varphi\right)(z) = \sum_{j=1}^d c_j z_j \int_0^1 \int_{\mathbb{T}^d} Df\big(|\varphi_j(\zeta,r)|^2\big)\,dm_d(\zeta)\,2rdr. 
	\end{equation}
	By linearity of $P$ and $D$, it is clear that \eqref{eq:fproj} holds for any polynomial $f$. Suppose that $f$ is continuously differentiable on $[0,C]$ for $C=(c_1+c_2+\cdots+c_d)^2$. Then \eqref{eq:fproj} holds for $f$ since both $f$ and $f'$ may be simultaneously uniformly approximated by polynomials on $[0,C]$. In particular, \eqref{eq:fproj} holds for $f(x)=(\delta+x)^{p/2-1}$ for $\delta>0$. By Fubini's theorem, we may let $\delta\to0^+$ when $p\geq1$ and obtain \eqref{eq:Plinear} from \eqref{eq:fproj}. 
\end{proof}
\begin{remark}
	We may replace the above polynomial approximation argument by an appeal to analytic continuation in the variable $p$ and the fact that the sequence $2n$ is not a Blaschke sequence in the right half-plane. The latter kind of argument is used in the proof of Theorem~\ref{thm:orthocex} below. 
\end{remark}

Note that we may interpret the integrals on the right-hand side of \eqref{eq:Plinear} as area integrals over the unit disc with respect to the variable $w = r z_j$. Let $A$ denote the Lebesgue measure of $\mathbb{C}$, normalized so that $A(\mathbb{D})=1$. The following result is pertinent to our analysis of the right-hand side of \eqref{eq:Plinear}. 
\begin{lemma}\label{lem:subharm} 
	Fix $2<p<\infty$. If $a>b>0$, then
	\[\int_{\mathbb{T}}\int_{\mathbb{D}} |aw+bz+c|^{p-2} \,dA(w)\,dm_1(z) < \int_{\mathbb{T}}\int_{\mathbb{D}} |az+bw+c|^{p-2}\,dA(w)\,dm_1(z)\]
	for every complex number $c$. 
\end{lemma}
\begin{proof}
	We begin by interchanging the order of integration and using rotational invariance to obtain 
	\begin{equation}\label{eq:abstrick} 
		\int_{\mathbb{T}}\int_{\mathbb{D}} |aw+bz+c|^{p-2} \,dA(w)\,dm_1(z) = \int_{\mathbb{D}}\int_{\mathbb{T}} \big||aw+b|z+c\big|^{p-2} \,dm_1(z)\,dA(w). 
	\end{equation}
	Notice that since $a>b$, the M\"{o}bius transformation
	\[w \mapsto \frac{aw+b}{a+bw}=\frac{w+b/a}{1+(b/a)w}\]
	maps the unit disc into itself. Hence 
	\begin{equation}\label{eq:Mest} 
		|aw+b|<|a+bw| 
	\end{equation}
	for every $w$ in $\mathbb{D}$. The function $z \mapsto |z + c|^{p-2}$ is subharmonic for each fixed complex number $c$. We can therefore use \eqref{eq:Mest} to conclude that
	\[\int_{\mathbb{T}} \big||aw+b|z+c\big|^{p-2}\,dm_1(z)<\int_{\mathbb{T}} \big||a+bw|z+c\big|^{p-2}\,dm_1(z).\]
	Integrating over $\mathbb{D}$ with respect to $w$ and using \eqref{eq:abstrick} twice, we obtain the stated inequality. 
\end{proof}

We are now ready to proceed with the proof of Theorem~\ref{thm:linear2}. 
\begin{proof}
	[Proof of Theorem~\ref{thm:linear2} for $2<p<\infty$] We may assume without loss of generality that $c_j>0$ for $j=1,2,\ldots,d$. If $d=1$ there is nothing to prove, so we also assume that $d\geq2$. We appeal to Theorem~\ref{thm:duality} (a) and Theorem~\ref{thm:Plinear} to conclude that $\varphi$ is a Hilbert point in $H^p(\mathbb{T}^d)$ if and only if
	\[\frac{p}{2}\sum_{j=1}^d c_j z_j \int_0^1 \int_{\mathbb{T}^d} |\varphi_j(z,r)|^{p-2}\,dm_d(z)\,2rdr = \lambda \sum_{j=1}^d c_j z_j\]
	for some constant $\lambda>0$. Of course, this is equivalent to the claim that 
	\begin{equation}\label{eq:intseq} 
		I_1 = I_2 = \cdots = I_d, 
	\end{equation}
	where we for $j=1,2,\ldots,d$ define
	\[I_j := \int_0^1 \int_{\mathbb{T}^d} |\varphi_j(z,r)|^{p-2}\,dm_d(z)\,2rdr.\]
	We will use a contrapositive argument, so assume that $a=c_1$ and $b=c_2$ for some $a>b>0$. By rotations and changing the order of integration, we find that 
	\begin{align*}
		I_1 &= \int_{\mathbb{T}^{d-2}} \int_{\mathbb{T}}\int_{\mathbb{D}}\Bigg|aw+bz+\sum_{j=2}^d c_j z_j\Bigg|^{p-2}\,dA(w)\,dm_1(z)\,dm_{d-2}(z) \\
		&<\int_{\mathbb{T}^{d-2}} \int_{\mathbb{T}} \int_{\mathbb{D}}\Bigg|az+bw+\sum_{j=2}^d c_j z_j\Bigg|^{p-2}\,dA(w)\,dm_1(z)\,dm_{d-2}(z) = I_2, 
	\end{align*}
	where we used Lemma~\ref{lem:subharm} for each $c = \sum_{j=2}^d c_j z_j$. It is now clear from \eqref{eq:intseq} that $\varphi$ cannot be a Hilbert point in $H^p(\mathbb{T}^d)$. 
\end{proof}

For the proof of Theorem~\ref{thm:linear2} in the case $p=\infty$, we require a well-known result, which will also be used in Section~\ref{sec:Khintchin} for $p < \infty$. Let $H^p_1(\mathbb{T}^d)$ be the subspace of $H^p(\mathbb{T}^d)$ comprised of $1$-homogeneous polynomials. 
\begin{lemma}\label{lem:P1} 
	The orthogonal projection $P_1 \colon H^2(\mathbb{T}^d) \to H^2_1(\mathbb{T}^d)$ extends to a contraction on $H^p(\mathbb{T}^d)$ for every $1 \leq p \leq \infty$. 
\end{lemma}
\begin{proof}
	The claim follows from the formula
	\[P_1 f (z) = \int_{\mathbb{T}} f(z_1 \zeta,z_2\zeta,\ldots,z_d \zeta)\,\overline{\zeta}\,dm_1(\zeta)\]
	and Minkowski's integral inequality. 
\end{proof}
\begin{proof}
	[Proof of Theorem~\ref{thm:linear2} for $p=\infty$] As in the case $2<p<\infty$, we assume without loss of generality that $c_j>0$ for $j=1,2,\ldots,d$. By Lemma~\ref{lem:P1}, we know that the projection from $H^\infty(\mathbb{T}^d)$ to $H_1^\infty(\mathbb{T}^d)$ is contractive, which implies that
	\[\|\varphi\|_{(H^\infty(\mathbb{T}^d))^\ast} = \|\varphi\|_{(H_1^\infty(\mathbb{T}^d))^\ast}.\]
	Noting that
	\[\|\varphi\|_{H^\infty(\mathbb{T}^d)} = \sum_{j=1}^d |c_j| = \sum_{j=1}^d c_j\]
	by our assumption that $c_j>0$ it is clear that $\|\varphi\|_{(H^\infty_1(\mathbb{T}^d))^\ast} = \max_{1 \leq j \leq d} c_j$. Hence we get from Theorem~\ref{thm:duality} (b) that $\varphi$ is a Hilbert point in $H^\infty(\mathbb{T}^d)$ if and only if
	\[\left(\max_{1 \leq j \leq d} c_j\right) \sum_{j=1}^d c_j = \sum_{j=1}^d c_j^2.\]
	By the assumption that $c_j>0$, we see that $\varphi$ is a Hilbert point in $H^\infty(\mathbb{T}^d)$ if and only if $c_1=c_2=\cdots=c_d$. 
\end{proof}
The conclusion of Theorem~\ref{thm:linear2} also holds if $d=2$ and $1 \leq p < 2$. To see this, it is sufficient to establish the following result. It replaces Lemma~\ref{lem:subharm} in the proof of Theorem~\ref{thm:linear2}, but the inequality goes in the reverse direction. 
\begin{lemma}\label{lem:d2} 
	Fix $1 \leq p < 2$. If $\varphi(z)=a z_1 + b z_2$ for $a>b>0$,then
	\[\int_0^1 \int_{\mathbb{T}^2} |\varphi_1(z,r)|^{p-2}\,dm_d(z)\,2rdr>\int_0^1 \int_{\mathbb{T}^2} |\varphi_2(z,r)|^{p-2}\,dm_d(z)\,2rdr.\]
\end{lemma}
\begin{proof}
	As in the proof of Lemma~\ref{lem:subharm}, we get that $|aw+b|<|a+bw|$ for every $w \in \mathbb{D}$. The statement now follows from the fact that $p-2<0$. 
\end{proof}
Based on Theorem~\ref{thm:linear1}, Theorem~\ref{thm:linear2} and Lemma~\ref{lem:d2}, we offer now the following. 
\begin{conjecture}\label{conj:12} 
	Suppose that $1 \leq p < 2$. If $\varphi(z)=\sum_{j=1}^d c_j z_j$ is a Hilbert point in $H^p(\mathbb{T}^d)$, then the nonzero coefficients of $\varphi$ all have the same modulus. 
\end{conjecture}

The conjecture is open for $d\geq3$. In the next section we will obtain some evidence in support of Conjecture~\ref{conj:12}.

\section{Dynamics of the nonlinear projection operators} \label{sec:dynamics}

It may be easier to understand the action of the nonlinear projection operator $\varphi \mapsto P(|\varphi|^{p-2}\varphi)$ if we normalize it in the following way: 
\begin{equation}\label{eq:Pop} 
	\mathscr{P}_p(\varphi) := \frac{P(|\varphi|^{p-2}\varphi)}{\|P(|\varphi|^{p-2}\varphi)\|_{H^2(\mathbb{T} ^d)}}. 
\end{equation}
Then $\mathscr{P}_p$ maps the unit sphere of $H^2_1(\mathbb{T}^d)$ into itself by Theorem~\ref{thm:Plinear} and $\varphi$ is a fixed point of $\mathscr{P}_p$ if and only if it is a Hilbert point in $H^p(\mathbb{T}^d)$ by Theorem~\ref{thm:duality}~(a).

Consider a $1$-homogeneous polynomial
\[\varphi_0(z) = \sum_{j=1}^d c_j z_j,\]
normalized such that $\|\varphi_0\|_{H^2(\mathbb{T}^d)}=1$. We define inductively $\varphi_{n+1}:=\mathscr{P}_p(\varphi_n)$ for every nonnegative integer $n$. By Theorem~\ref{thm:duality} (a), we know that $\varphi_{n+1}=\varphi_n$ if and only if $\varphi_n$ is a Hilbert point in $H^p(\mathbb{T}^d)$. We let $c_j^{(n)}$ denote the coefficient of $\varphi_n$ at $z_j$ (so that $c_j=c_j^{(0)}$). What can we say about the behaviour of these coefficients when $n\to \infty$? We begin with two obvious conclusions, which follow at once from Theorem~\ref{thm:Plinear}. 
\begin{enumerate}
	\item[(i)] If $c_j=0$, then $c_j^{(n)}=0$ for every $n\geq0$. 
	\item[(ii)] If $c_j\neq0$, then $\arg(c_j^{(n)})=\arg(c_j)$ for every $n\geq0$. 
\end{enumerate}
For simplicity, we shall in what follows assume that $c_j>0$ for every $j=1,2,\ldots,d$. Next, let us compare two coefficients. 
\begin{enumerate}
	\item[(iii)] If $c_j=c_k$, then $c_j^{(n)} = c_k^{(n)}$ for every $n\geq0$. This follows at once from Theorem~\ref{thm:Plinear} and symmetry. 
\end{enumerate}
To see what happens when $c_j\neq c_k$, we will now establish a result that complements Lemma~\ref{lem:subharm} by giving an inequality in the opposite direction. While the proof of Lemma~\ref{lem:subharm} relied crucially on an argument involving subharmonicity, the next result follows from a purely geometric consideration.

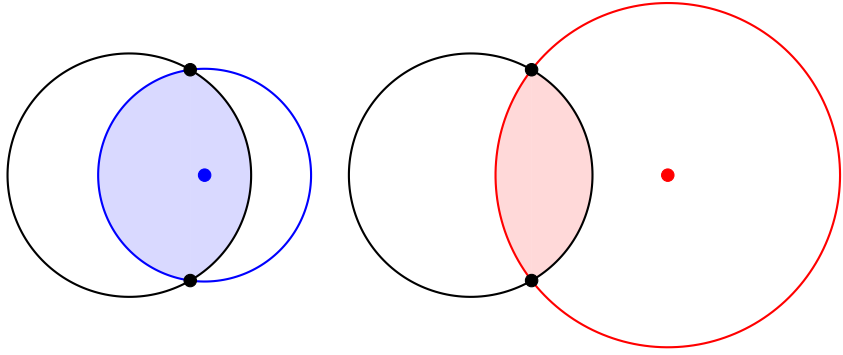
\begin{figure}
	\centering
	\begin{subfigure}{0.44\textwidth}
		\centering
		\begin{tikzpicture}
			\begin{axis}[
				axis equal image,
				axis lines = none,
				xmin = -1-0.1,
				xmax = (sqrt(5)+1)/2+sqrt(2)+0.1,
				ymin = -sqrt(2)-0.1,
				ymax = sqrt(2)+0.1,
				every axis x label/.style={
					at={(ticklabel* cs:1.025)},
					anchor=west,
				},
				every axis y label/.style={
					at={(ticklabel* cs:1.025)},
					anchor=south,
				},
				ticks=none,%
				axis line style={->}]
				
				\addplot[domain=(sqrt(5)-1)/2-(sqrt(5)-1)/sqrt(2):1/2, thin, 
				samples=200, draw opacity=0, name path=c1u] 
				{sqrt(((sqrt(5)-1)/sqrt(2))^2-(x-(sqrt(5)-1)/2)^2)};
				\addplot[domain=(sqrt(5)-1)/2-(sqrt(5)-1)/sqrt(2):1/2, thin, 
				samples=200, draw opacity=0, name path=c1l] 
				{-sqrt(((sqrt(5)-1)/sqrt(2))^2-(x-(sqrt(5)-1)/2)^2)};
				\addplot[fill=blue, fill opacity=0.15] fill between[of=c1u and 
				c1l];
				
				\addplot[domain=1/2:1, thin, samples=200, draw opacity=0, name 
				path=c2u] {sqrt(1-x^2)};
				\addplot[domain=1/2:1, thin, samples=200, draw opacity=0, name 
				path=c2l] {-sqrt(1-x^2)};
				\addplot[fill=blue, fill opacity=0.15] fill between[of=c2u and 
				c2l];
				
				\addplot[domain=360:0,name path=uc,samples=200,thick] 
				({cos(x)},{sin(x)});
				\addplot[domain=360:0,name path=c1,samples=200,color=blue,thick]
({(sqrt(5)-1)/2+(sqrt(5)-1)/sqrt(2)*cos(x)},{(sqrt(5)-1)/sqrt(2)*sin(x)});
				
				\node[circle, draw, color=blue, fill=blue, scale=0.5, opacity=1] 
at (0.6180339887,0){};
				\node[circle, draw, color=black, fill=black, scale=0.5, 
opacity=1] at (0.5,0.8660254038){};
				\node[circle, draw, color=black, fill=black, scale=0.5, 
opacity=1] at (0.5,-0.8660254038){};
			\end{axis} 
		\end{tikzpicture}
	\end{subfigure} \hspace{-0.1\textwidth} 
\begin{subfigure}{0.54\textwidth}
		\begin{tikzpicture}
			\begin{axis}[
				axis equal image,
				axis lines = none,
				xmin = -1-0.1,
				xmax = (sqrt(5)+1)/2+sqrt(2)+0.1, 
				ymin = -sqrt(2)-0.1,
				ymax = sqrt(2)+0.1,
				every axis x label/.style={
					at={(ticklabel* cs:1.025)},
					anchor=west,
				},
				every axis y label/.style={
					at={(ticklabel* cs:1.025)},
					anchor=south,
				},
				ticks=none,%
				axis line style={->}]
			
				\addplot[domain=(sqrt(5)+1)/2-sqrt(2):1/2, thin, samples=200, 
draw opacity=0, name path=c1u] {sqrt(2-(x-(sqrt(5)+1)/2)^2)};
				\addplot[domain=(sqrt(5)+1)/2-sqrt(2):1/2, thin, samples=200, 
draw opacity=0, name path=c1l] {-sqrt(2-(x-(sqrt(5)+1)/2)^2)};
				\addplot[fill=red, fill opacity=0.15] fill between[of=c1u and 
c1l];
			
				\addplot[domain=1/2:1, thin, samples=200, draw opacity=0, name 
path=c2u] {sqrt(1-x^2)};
				\addplot[domain=1/2:1, thin, samples=200, draw opacity=0, name 
path=c2l] {-sqrt(1-x^2)};
				\addplot[fill=red, fill opacity=0.15] fill between[of=c2u and 
c2l];
			
				\addplot[domain=360:0,name path=uc,samples=200,thick] 
({cos(x)},{sin(x)});
				\addplot[domain=360:0,name path=c1,samples=200,color=red,thick] 
({(sqrt(5)+1)/2+sqrt(2)*cos(x)},{sqrt(2)*sin(x)});
				
				\node[circle, draw, color=red, fill=red, scale=0.5, opacity=1] 
at (1.6180339887,0){};
				\node[circle, draw, color=black, fill=black, scale=0.5, 
opacity=1] at (0.5,0.8660254038){};
				\node[circle, draw, color=black, fill=black, scale=0.5, 
opacity=1] at (0.5,-0.8660254038){};
			\end{axis} 
		\end{tikzpicture}
	\end{subfigure}
	\caption{The areas {\color{blue}$A_1(r)$} and {\color{red}$A_2(r)$} in the 
proof of Lemma~\ref{lem:circles}, for {\color{red}$x=\frac{1+\sqrt{5}}{2}$} and 
{\color{blue} $\frac{1}{x}=\frac{\sqrt{5}-1}{2}$}. The black circle is the unit 
circle. We have chosen $r=\sqrt{2}$ or, equivalently, $\theta=\frac{\pi}{3}$.}  
	\label{fig:circles}
\end{figure}

\begin{lemma}\label{lem:circles} 
	Fix $1 \leq p<\infty$. If $a>b>0$, then
	\[a\int_{\mathbb{T}}\int_{\mathbb{D}} |aw+bz+c|^{p-2} \,dA(w)\,dm_1(z) > b\int_{\mathbb{T}}\int_{\mathbb{D}} |az+bw+c|^{p-2}\,dA(w)\,dm_1(z)\]
	for every complex number $c$. 
\end{lemma}
\begin{proof}
	Replacing $c$ by $c/b$, we assume without loss of generality that $a=x>1$ and $b=1$. Set $\Phi(r,z,c)=|rz+c|^{p-2}$. As in the proof of Lemma~\ref{lem:subharm}, we write 
	\begin{align*}
		\int_{\mathbb{T}}\int_{\mathbb{D}} |xz+w+c|^{p-2} \,dA(w)\,dm_1(z) &= \int_{\mathbb{T}} \int_{\mathbb{D}} \Phi(|1+xw|,z,c)\,dA(w)\,dm_1(z), \\
		\int_{\mathbb{T}}\int_{\mathbb{D}} |z+xw+c|^{p-2} \,dA(w)\,dm_1(z) &= \int_{\mathbb{T}} \int_{\mathbb{D}} \Phi(|x+w|,z,c)\,dA(w)\,dm_1(z). 
	\end{align*}
	For fixed $z$ on $\mathbb{T}$ and complex number $c$, we consider
	\[I_1(x) := \int_{\mathbb{D}} \Phi(|1+xw|,z,c)\,dA(w) \qquad \text{and} \qquad I_2(x) := \int_{\mathbb{D}} \Phi(|x+w|,z,c)\,dA(w)\]
	For $0 \leq r \leq x+1$, we define
	\[A_1(r):=A(\{w\in \mathbb{D}\,:\, |1+xw|\leq r \}) \quad \text{and} \quad A_2(r):=A(\{w\in \mathbb{D}\,:\, |x+w|\leq r \}).\]
	By symmetry, we note that $A_1(r)$ is equal to the area of the intersection of the disc $\mathbb{D}(1/x,r/x)$ and the unit disc $\mathbb{D}$ and, similarly, that $A_2(r)$ is equal to the area of the intersection of the discs $\mathbb{D}(x,r)$ and $\mathbb{D}$. See Figure~\ref{fig:circles}. Since $A_2(r)=0$ for $r<x-1$, we rewrite the integrals as
	\[I_1(x)=\int_0^{x+1} \Phi(r,z,c) A_1'(r) \,dr \qquad \text{and} \qquad I_2(x) =\int_{x-1}^{x+1} \Phi(r,z,c) A_2'(r) \,dr, \]
	by polar coordinates and change of variables. Since $\Phi$ is nonnegative, we are done if we can prove that $A_2'(r) \leq x A_1'(r)$ for $x-1 \leq r \leq 1$. We restrict $r$ to this interval henceforth.
	
	The unit circle intersects the circles $|1+wx|=r$ and $|x+w|=r$ in the same two points $e^{ \pm i\theta}$ for some $0<\theta<\pi$. We find it convenient to consider $\theta$ a function of $r$. Let $\ell_1$ and $\ell_2$ denote the arc length of the part of the circles intersecting the unit disc, respectively. Then
	\[\ell_1(r) = 2 \frac{r}{x} \arctan{\left(\frac{\sin{\theta}}{\cos{\theta}+1/x}\right)} \qquad\text{and}\qquad \ell_2(r) = 2 r \arctan{\left(\frac{\sin{\theta}}{\cos{\theta}+x}\right)},\]
	where $\arctan$ takes values in $[0,\pi]$. Inspecting Figure~\ref{fig:circles} again, we find that
	\[A_1(r) = \pi\left(\frac{x-1}{x}\right)^2 + \int_{\frac{x-1}{x}}^{\frac{r}{x}} \ell_1(xs)\,ds \qquad\text{and}\qquad A_2(r) = \int_{x-1}^r \ell_2(s)\,ds,\]
	when $x-1<r<x+1$, from which we see that
	\[A'_1(r) = \frac{2r}{x^2}\arctan{\frac{\sin \theta}{\cos{\theta} + 1/x}} \qquad \text{and} \qquad A'_2(r) = 2r\arctan{\frac{\sin{\theta}}{\cos{\theta} + x}}.\]
	We shall now fix $x-1<r<x+1$, or equivalently $0<\theta<\pi$. To establish the desired estimate $A_2'(r) \leq x A_1'(r)$ it is enough to check that $F_\theta(x)>0$ for $x>1$, where
	\[F_\theta(x) = \arctan{\left(\frac{\sin \theta}{\cos \theta + 1/x}\right)}-x \arctan{\left(\frac{\sin \theta}{\cos \theta + x}\right)}.\]
	Since $F_\theta(1)=0$, we compute
	\[F_\theta'(x) = \frac{(x+1)\sin{\theta}}{1+2x\cos{\theta}+x^2} - \arctan{\left(\frac{\sin \theta}{\cos \theta + x}\right)} \geq \frac{(x+1)\sin{\theta}}{1+2x\cos{\theta}+x^2} -\frac{\sin \theta}{\cos \theta + x},\]
	using the estimate $\arctan(y)\leq y$ for $y\geq0$. Since
	\[\frac{(x+1)\sin{\theta}}{1+2x\cos{\theta}+x^2} -\frac{\sin \theta}{\cos{\theta} + x} = \frac{\sin{\theta}(1-\cos{\theta})(x-1)}{(1+2x\cos{\theta}+x^2)(\cos{\theta} + x)}\]
	and conclude that $F_\theta'(x)>0$, which completes the proof. 
\end{proof}
We may now make the following additional assertion. 
\begin{enumerate}
	\item[(iv)] If $c_j>c_k$, then $c_j^{(n)}>c_k^{(n)}$ for every $n\geq0$. This is a consequence of Theorem~\ref{thm:Plinear} and Lemma~\ref{lem:circles}. 
\end{enumerate}

Combining the assertions (i)---(iv) with Theorem~\ref{thm:Plinear} and Lemma~\ref{lem:subharm}, we may obtain the following result. 
\begin{theorem}\label{thm:iterative2} 
	Fix $2<p<\infty$. Suppose that $\varphi_0(z) = \sum_{j=1}^d c_j z_j$ is an arbitrary point in the unit sphere of $H^2_1(\mathbb{T}^d)$ and that $c_j\neq 0$ for $j=1,2,\ldots, d$. Then
	\[\lim_{n\to\infty}(\mathscr{P}_p^n \varphi_0)(z) = \frac{1}{\sqrt{d}} \sum_{j=1}^d \frac{c_j}{|c_j|} z_j. \]
\end{theorem}
\begin{proof}
	We may assume without loss of generality that $c_1 \geq c_j>0$ for $j=2,3,\ldots,d$. By (iv), this ordering will persist under iterations by $\mathscr{P}_p$ so that we will have $c_1^{(n)}\geq c_j^{(n)}$ for all $n$ and $j=2,3,\ldots, d$. In particular, this implies that $c_1^{(n)}\geq d^{-1/2}$ by the normalization. The crux of the proof will be to show that $n\mapsto c_1^{(n)}$ is a strictly decreasing sequence whenever $c_1^{(0)}>d^{-1/2}$.
	
	We begin by showing how to conclude once we know that $n\mapsto c_1^{(n)}$ is strictly decreasing. If
	\[c_1^{(\infty)}:=\lim_{n\to \infty} c_1^{(n)}=\frac{1}{\sqrt{d}},\]
	then we are done because the ordering of the coefficients persists under iterations. We will next rule out the possibility that $c_1^{(\infty)}>d^{-1/2}$. If this were the case, we could by compactness find a subsequence $n_k$ and coefficients $c_j^{(\infty)}$ such that
	\[c_j^{(\infty)}=\lim_{k\to \infty} c_j^{(n_k)}\]
	for $j=1,2,\ldots, d$. Clearly, the ordering persists in the limit so that $c_1^{(\infty)}\geq c_j^{(\infty)}$ for every $j=2,3,\ldots, d$. If we now start iterating from
	\[\varphi_{\infty}(z):=\sum_{j=1}^d c_j^{(\infty)} z_j,\]
	then the largest coefficient of the iterates will again be a strictly decreasing sequence. However, this would violate the fact that the coefficients of $\mathscr{P}_p(\varphi)$ for $\varphi$ in the unit sphere of $H^2_1(\mathbb{T}^d)$ depend continuously on the coefficients of $\varphi$.
	
	It remains to show that $n\mapsto c_1^{(n)}$ is strictly decreasing when $c_1^{(0)}>d^{-1/2}$. Then there exists a $j_0$ such that $c_{j_0}^{(0)}<d^{-1/2}<c_1^{(0)}$. By (iv) and induction on $n$, we have then $c_{j_0}^{(n)}<c_1^{(n)}$ for all nonnegative integers $n$. Now invoking Lemma~\ref{lem:subharm} and taking into account Theorem~\ref{thm:Plinear}, we see that the ratios $c_1^{(n)}/c_{j}^{(n)}$ are nonincreasing for $j=2,3,\ldots,d$. This allows us to draw the desired conclusion because we have seen that at least one of these sequences of ratios is strictly decreasing. 
\end{proof}
\begin{remark}
	Theorem~\ref{thm:linear1} and Theorem~\ref{thm:linear2} can be obtained as direct corollaries to Theorem~\ref{thm:iterative2} through Theorem~\ref{thm:duality} (a). 
\end{remark}

Suppose that we start iterating from $\varphi_0(z) = az_1+bz_2$ for $|a|^2+|b|^2=1$ when $1 \leq p < 2$. Replacing Lemma~\ref{lem:subharm} by Lemma~\ref{lem:d2}, we obtain by similar considerations as above the following conclusions. 
\begin{itemize}
	\item If $|a|=|b|=1/\sqrt{2}$, then $\varphi_\infty(z) = \varphi_0(z)$. 
	\item If $|a|>|b|$, then $\varphi_\infty(z)= \frac{a}{|a|} z_1$. 
	\item If $|a|<|b|$, then $\varphi_\infty(z)= \frac{b}{|b|} z_2$. 
\end{itemize}
The key difference between the cases $1\leq p<2$ and $2<p<\infty$ is that if $a>b>0$, then the sequence $a^{(n)}$ is strictly increasing in the former and decreasing in the latter. 

Consider now $\varphi_0$ in the unit sphere of $H^2_1(\mathbb{T}^d)$ and apply the nonlinear projection operator \eqref{eq:Pop} for $d\geq3$ and $1 \leq p < \infty$. Repeating the reasoning of the first part of the proof of Theorem~\ref{thm:iterative2}, we see that to extend Theorem~\ref{thm:Plinear} to the range $1\leq p<2$, it would suffice to show that in this case, the largest coefficient of the iterates is strictly increasing. We have performed some numerical experiments when $d=3$ and $1 \leq p < 2$, picking many random polynomials from $H^2_1(\mathbb{T}^3)$ as initial point and applying the iteration. A representative example (with $p=1$) can be found in Table~\ref{tab:iterations}.

Table~\ref{tab:iterations} reveals another difference between $1 \leq p < 2$ and $2 < p < \infty$, since the ratio $n \mapsto a^{(n)}/c^{(n)}$ is not monotone. In this example, the ratio decreases in the first two iterations and then increases thereafter. This indicates that the case $1 \leq p < 2$ is more subtle, since it is not sufficient to consider the pairwise interaction of coefficients under the iterations. 
\begin{question}\label{que:bigcoeff} 
	Suppose that $\varphi_0(z) = \sum_{j=1}^d c_j z_j$ is in the unit sphere of $H^2_1(\mathbb{T}^d)$ and that $c_1 > c_j \geq 0$ for every $j=2,3,\ldots,d$. Is it true that
	\[\lim_{n\to\infty} (\mathscr{P}_p^n \varphi_0)(z) = z_1\]
	whenever $1 \leq p < 2$? 
\end{question}

It follows from the above discussion that a positive answer to Question~\ref{que:bigcoeff} would lead to a proof of Conjecture~\ref{conj:12}.
\begin{table}
	\begin{tabular}
		{r|c|c|c} $n$ & $a^{(n)}$ & $b^{(n)}$ & $c^{(n)}$ \\
		\hline $0$ & $0.7256$ & $0.6766$ & $0.1251$ \\
		$1$ & $0.7577$ & $0.6346$ & $0.1520$ \\
		$2$ & $0.8259$ & $0.5413$ & $0.1576$ \\
		$3$ & $0.9191$ & $0.3762$ & $0.1175$ \\
		$4$ & $0.9742$ & $0.2152$ & $0.0686$ \\
		$5$ & $0.9931$ & $0.1120$ & $0.0359$ \\
		$6$ & $0.9982$ & $0.0566$ & $0.0182$ \\
		$7$ & $0.9996$ & $0.0284$ & $0.0091$ \\
		$8$ & $0.9999$ & $0.0142$ & $0.0046$ 
	\end{tabular}
	\caption{Iterations of \eqref{eq:Pop} with $p=1$ starting from $\varphi_0(z) = a^{(0)}z_1 + b^{(0)}z_2+c^{(0)}z_3$, computed numerically to precision $10^{-5}$.} 
\label{tab:iterations} \end{table}

It is natural to ask how the dynamics of $\mathscr{P}_p$ may be in a more general situation. Notice however that $\mathscr{P}_p$ is not well defined on the unit sphere of $H^2(\mathbb{T}^d)$ when $p>2$, so that it is not clear how to proceed in full generality. One could imagine modifiying the definition of $\mathscr{P}_p$ or restricting again to some submanifold of the unit sphere of $H^2(\mathbb{T}^d)$ that is preserved by $\mathscr{P}_p$, such as that consisting of $m$-homogeneous polynomials. It would be interesting to know in which generality what was observed above may hold, namely that inner functions are attracting fixed points for $1\leq p < 2$ and repelling fixed points for $2<p<\infty$.

\section{Khintchin's inequality for Steinhaus variables} \label{sec:Khintchin} 
We will now see how Theorem~\ref{thm:linear2} (and Theorem~\ref{thm:Plinear}) can be applied to give a proof of the sharp Khintchin inequality \eqref{eq:Khintchin} in the range $2<p<\infty$. Recall that a Steinhaus random variable by definition is uniformly distributed on $\mathbb{T}$ with respect to the Lebesgue arc length measure. Hence, if $(z_j)_{j\geq1}^d$ is a sequence of independent Steinhaus variables and $(c_j)_{j\geq1}^d$ are complex numbers, then
\[\mathbb{E}\Bigg|\sum_{j=1}^d c_j z_j \Bigg|^p = \int_{\mathbb{T}^d} |c_1 z_1 + \cdots + z_d z_d|^p\,dm_d(z) = \|\varphi\|_{H^p(\mathbb{T}^d)}^p.\]

The novelty of our proof of Khintchin's inequality is that we avoid using bisubharmonic functions as was done in \cite{BC02}. It may be observed, however, that subharmonicity plays an essential role, namely in the proof of Lemma~\ref{lem:subharm}.

Our proof begins with Lemma~\ref{lem:Khintchin}, where we consider critical points of the functional 
\begin{equation}\label{eq:khinfunc} 
	\mathscr{K}_p(c) := \|c_1z_1+\cdots+c_dz_d\|_{H^p(\mathbb{T}^d)} 
\end{equation}
defined for $c=(c_1,\ldots , c_d)$ in the unit sphere of $\mathbb{C}^d$. Recall that $H_1^p(\mathbb{T}^d)$ is the $d$-dimensional subspace of $H^p(\mathbb{T}^d)$ comprised of $1$-homogeneous polynomials. 
\begin{proof}
	[Proof of Lemma~\ref{lem:Khintchin}] Fix $2<p<\infty$. For $c$ in the unit sphere of $\mathbb{C}^d$, let $\varphi$ denote the associated $1$-homogeneous polynomial. By the Lagrange multiplier theorem, any critical point $(c_1,\ldots,c_d)$ of the functional \eqref{eq:khinfunc} satisfies
	\[\nabla \|\varphi\|_{H^p(\mathbb{T}^d)} = \lambda \nabla \|\varphi\|_{H^2(\mathbb{T}^d)}\]
	for some constant $\lambda$. This means that the complex tangent space to the closed ball in $H_1^p(\mathbb{T}^d)$ centered at the origin with radius $\|\varphi\|_{H^p(\mathbb{T}^d)}$ at the point $\varphi$ is the same as the complex tangent space to the closed unit ball in $H_1^2(\mathbb{T}^d)$ at the point $\varphi$. But this condition means that for any $1$-homogeneous polynomial $f$ such that $\langle f, \varphi \rangle = 0$ we have that 
	\begin{equation}\label{eq:1hom} 
		\|\varphi +f\|_{H^p(\mathbb{T}^d)} \geq \|\varphi\|_{H^p(\mathbb{T}^d)}. 
	\end{equation}
	By Lemma~\ref{lem:P1}, we get that \eqref{eq:1hom} holds for all $1$-homogeneous polynomials $f$ satisfying $\langle f,\varphi \rangle=0$ if and only if $\varphi$ is a Hilbert point in $H^p(\mathbb{T}^d)$. 
\end{proof}
\begin{remark}
	The above proof is a finite-dimensional version of the argument used to establish Corollary~\ref{cor:hyperplane}, where we saw that $T_p\cap H^2(\mathbb{T}^d) = T_2 \cap H^p(\mathbb{T}^d)$ at a Hilbert point in $H^p(\mathbb{T}^d)$. 
\end{remark}

By Lemma~\ref{lem:Khintchin} and Theorem~\ref{thm:linear2}, we know that to get the optimal upper and lower bounds in Khintchin's inequality when $2<p<\infty$, we only need to investigate the $1$-homogeneous polynomials for which all the nonzero coefficients have the same modulus. Hence we require the following result. 
\begin{lemma}\label{lem:kinc} 
	If $2< p < \infty$, then
	\[d \mapsto \left\|\frac{1}{\sqrt{d}} \sum_{j=1}^d z_j \right\|_{H^p(\mathbb{T}^d)}\]
	is strictly increasing for $d\geq1$. 
\end{lemma}
\begin{proof}
	Fix $2<p<\infty$. For $0 \leq t \leq 1$, we consider the function
	\[\varphi_t(z): = \left(\frac{1-t}{d}+\frac{t}{d+1}\right)^{1/2} (z_1+\cdots+z_d) + \left(\frac{t}{d+1}\right)^{1/2} z_{d+1}\]
	and define $\Phi(t) := \|\varphi_t\|_{H^p(\mathbb{T}^d)}^p$. We need to prove that $\Phi(0)< \Phi(1)$ which we will do by showing that $\Phi'(t)>0$ for $0<t<1$. Writing $|\varphi_t|^p = (\varphi_t \overline{\varphi_t})^{p/2}$ and using the chain rule and the product rule, we find that
	\[\Phi'(t) = \frac{p}{2} \int_{\mathbb{T}^{d+1}} |\varphi_t(z)|^{p-2} \varphi_t(z) \cdot 2\frac{d}{dt} \varphi_t(\overline{z})\,dm_{d+1}(z).\]
	To proceed, we first note that we may replace $|\varphi_t|^{p-2} \varphi_t$ by $P(|\varphi_t|^{p-2} \varphi_t)$ by orthogonality. Next we compute
	\[2 \frac{d}{dt} \varphi_t(\overline{z}) = -\frac{1}{d(d+1)}\left(\frac{1-t}{d}+\frac{t}{d+1}\right)^{-1/2}(\overline{z_1} +\cdots+\overline{z_d}) + \left(\frac{1}{(d+1)t}\right)^{1/2}\overline{z_{d+1}}.\]
	By Theorem~\ref{thm:Plinear}, we find that 
	\begin{align*}
		\Phi'(t) = \frac{p^2}{4(d+1)} \Bigg(&-\int_0^1 \int_{\mathbb{T}^{d+1}}|(\varphi_t)_1(\zeta,r)|^{p-2} 2 r dr\,dm_{d+1}(\zeta) \\
		&+\int_0^1 \int_{\mathbb{T}^{d+1}}|(\varphi_t)_{d+1}(\zeta,r)|^{p-2} 2 r dr\,dm_{d+1}(\zeta)\Bigg). 
	\end{align*}
	Since $p>2$ and $\frac{1-t}{d}+\frac{t}{d+1}>\frac{t}{d+1}$ for $0<t<1$, we find that $\Phi'(t)>0$ by using Lemma~\ref{lem:subharm} as in the proof of Theorem~\ref{thm:linear2}. 
\end{proof}
\begin{theorem}
	[Khintchin's inequality \cite{BC02,KK01}] Fix $2<p<\infty$. We have
	\[\Bigg\|\sum_{j=1}^d c_j z_j \Bigg\|_{H^p(\mathbb{T}^d)} \leq {\Gamma\left(1+\frac{p}{2}\right)}^\frac{1}{p} \Bigg(\sum_{j=1}^d |c_j|^2\Bigg)^\frac{1}{2}\]
	for all complex numbers $c_1,\ldots , c_d$. The constant ${\Gamma\left(1+\frac{p}{2}\right)}^\frac{1}{p}$ is optimal. 
\end{theorem}
\begin{proof}
	By Lemma~\ref{lem:Khintchin}, Theorem~\ref{thm:linear2}, and Lemma~\ref{lem:kinc}, we obtain that the asserted inequality holds with optimal constant equal to
	\[\lim_{d\to \infty}\left\|\frac{1}{\sqrt{d}} \sum_{j=1}^d z_j \right\|_{H^p(\mathbb{T}^d)}={\Gamma\left(1+\frac{p}{2}\right)}^\frac{1}{p}.\]
	The limit can be evaluated by the central limit theorem, since the independent complex-valued random variables $(z_j)_{j\geq1}$ have mean $0$ and variance $1$. 
\end{proof}
\begin{remark}
	The proof of Khintchin's inequality for $1 \leq p<2$ in \cite{KK01} requires rather technical estimates. This indicates why Conjecture~\ref{conj:12} could be more difficult to establish compared to Theorem~\ref{thm:linear2}, since a positive answer to the former would simplify the proof of Khintchin's inequalty for $1 \leq p < 2$ substantially. 
\end{remark}

\section{A Hilbert point in \texorpdfstring{$H^4(\mathbb{T}^3)$}{H4(T3)}} \label{sec:example} 
We have so far devoted our attention to two classes of Hilbert points. If $\varphi = CI$ for a constant $C\neq0$ and an inner function $I$, then $\varphi$ is a Hilbert point for every $1 \leq p \leq \infty$ by Corollary~\ref{cor:inneryes}. If $\varphi$ is a $1$-homogeneous polynomial, then it follows from Theorem~\ref{thm:linear1} and Theorem~\ref{thm:linear2} that if $\varphi$ is a Hilbert point in $H^p(\mathbb{T}^d)$ for \emph{some} $2<p\leq\infty$, then it is a Hilbert point in $H^p(\mathbb{T}^d)$ for every $1 \leq p \leq \infty$. Conjecture~\ref{conj:12} implies that the same statement should hold if $2<p\leq \infty$ is replaced by $1 \leq p < 2$.

The purpose of the present section is to demonstrate that in general, when $d\geq2$, the Hilbert points depend on $p$. We begin with the following result, which is inspired by \cite[Ex.~3.4]{BOS21}. 
\begin{theorem}\label{thm:3ortho} 
	The function
	\[\varphi(z) = c_1 z_1^3 + c_2 z_2^3 + c_3 z_3^3 + c_4 z_1 z_2 z_3\]
	is a Hilbert point in $H^4(\mathbb{T}^3)$ if and only if the nonzero coefficients of $\varphi$ all have the same modulus. 
\end{theorem}
\begin{proof}
	We will use Theorem~\ref{thm:duality} (a), and begin by expanding 
	\begin{align*}
		|\varphi(z)|^2 = \|\varphi\|_{H^2(\mathbb{T}^3)}^2 &+ c_1 z_1^3 \left(\overline{c_2 z_2^3} + \overline{c_3 z_3^3} + \overline{c_4 z_1 z_2 z_3} \right) \\
		&+ c_2 z_2^3 \left(\overline{c_1 z_1^3} + \overline{c_3 z_3^3} + \overline{c_4 z_1 z_2 z_3} \right) \\
		&+ c_3 z_3^3 \left(\overline{c_1 z_1^3} + \overline{c_2 z_2^3} + \overline{c_4 z_1 z_2 z_3} \right) \\
		&+ c_1c_2c_3 z_1z_2z_3 \left(\overline{c_1 z_1^3} + \overline{c_2 z_2^3} + \overline{c_3 z_3^3}\right). 
	\end{align*}
	Hence, we find that
	\[P\left(|\varphi|^2 \varphi\right) = \sum_{j=1}^3 c_j z_j \left(2\|\varphi\|_{H^2(\mathbb{T}^3)}^2-|c_j|^2\right) + c_4 z_1z_2z_3 \left(2\|\varphi\|_{H^2(\mathbb{T}^3)}^2-|c_4|^2\right),\]
	from which we easily deduce that the solutions of the equation $P(|\varphi|^2 \varphi) = \lambda \varphi$ have the stated form. 
\end{proof}

If $c_4=0$, then the conclusion of Theorem~\ref{thm:3ortho} can be obtained directly from Theorem~\ref{thm:linear1} and Theorem~\ref{thm:linear2} by substituting $\zeta_1:=z_1^3$, $\zeta_2:=z_2^3$ and $\zeta_3:=z_3^3$. However, the same argument shows that $\varphi$ is a Hilbert point also in $H^p(\mathbb{T}^d)$ for every $1\leq p\leq\infty$. We now turn to the main result of this section, where we see that putting $c_3=0$ leads to a completely different situation. 
\begin{theorem}\label{thm:orthocex} 
	The function $\varphi(z) = z_1^3+z_2^3 + z_1z_2z_3$ is a Hilbert point in $H^2(\mathbb{T}^3)$ and $H^4(\mathbb{T}^3)$, but not in $H^p(\mathbb{T}^3)$ for any $4<p\leq \infty$. 
\end{theorem}

Numerical evidence (see Figure~\ref{fig:plot}) suggests that this $\varphi$ is a Hilbert point in $H^p(\mathbb{T}^d)$ only when $p=2,4$. Unfortunately, we are only able to verify analytically that there is possibly a finite number of $p$ in $[1,4)\setminus \{2\}$ for which $\varphi$ is a Hilbert point in $H^p(\mathbb{T}^d)$.

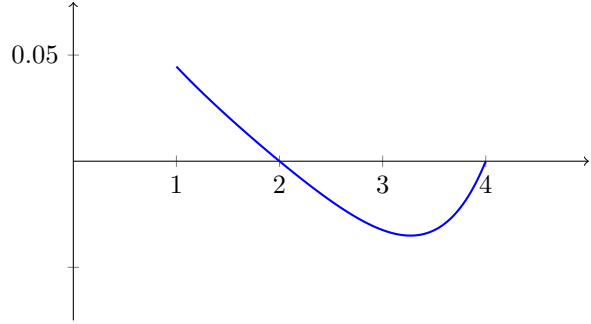
\begin{figure}
	\centering
	\begin{tikzpicture}
		\begin{axis}
			[axis equal image,
			axis lines=middle,
			axis line style=thin,
			xmin=0,
			xmax=5,
			xtick={1,2,3,4},
			xticklabels={1,2,3,4,},
			ymin=-0.75,
			ymax=0.75,
			ytick={-0.5,0.5},
			yticklabels={,0.05},
			every axis x label/.style={ at={(ticklabel* cs:1.025)}, 
				anchor=west,},
			every axis y label/.style={ at={(ticklabel* cs:1.025)}, 
				anchor=south,},
			axis line style={->}, y post scale = 2.0601]
			\addplot[thick, color=blue]
			coordinates{
			(1.000000000000000, 0.44602043138657)
			(1.010000000000000, 0.44089350144288)
			(1.020000000000000, 0.43579271603170)
			(1.030000000000000, 0.43070942074569)
			(1.040000000000000, 0.42564542402796)
			(1.050000000000000, 0.42060549017583)
			(1.060000000000000, 0.41552983110236)
			(1.070000000000000, 0.41052868000383)
			(1.080000000000000, 0.40554647421709)
			(1.090000000000000, 0.40058286872343)
			(1.100000000000000, 0.39563759459941)
			(1.110000000000000, 0.39071035018119)
			(1.120000000000000, 0.38580083537391)
			(1.130000000000000, 0.38090876581329)
			(1.140000000000000, 0.37603344707577)
			(1.150000000000000, 0.37117585628682)
			(1.160000000000000, 0.36633446351988)
			(1.170000000000000, 0.36150943367353)
			(1.180000000000000, 0.35669919830104)
			(1.190000000000000, 0.35190726473219)
			(1.200000000000000, 0.34712749095253)
			(1.210000000000000, 0.34236750572400)
			(1.220000000000000, 0.33762031887734)
			(1.230000000000000, 0.33288733429925)
			(1.240000000000000, 0.32816779611826)
			(1.250000000000000, 0.32346725857506)
			(1.260000000000000, 0.31877815385140)
			(1.270000000000000, 0.31410211909615)
			(1.280000000000000, 0.30944259331847)
			(1.290000000000000, 0.30479212222756)
			(1.300000000000000, 0.30015708573791)
			(1.310000000000000, 0.29553513231651)
			(1.320000000000000, 0.29092555391862)
			(1.330000000000000, 0.28632943817354)
			(1.340000000000000, 0.28174638031662)
			(1.350000000000000, 0.27717345217066)
			(1.360000000000000, 0.27261341059549)
			(1.370000000000000, 0.26806562135818)
			(1.380000000000000, 0.26352930221567)
			(1.390000000000000, 0.25900447162013)
			(1.400000000000000, 0.25449097496257)
			(1.410000000000000, 0.24998791380908)
			(1.420000000000000, 0.24549709853965)
			(1.430000000000000, 0.24101657913875)
			(1.440000000000000, 0.23654665439391)
			(1.450000000000000, 0.23208721111868)
			(1.460000000000000, 0.22763810847490)
			(1.470000000000000, 0.22319921417228)
			(1.480000000000000, 0.21877011775202)
			(1.490000000000000, 0.21435137959401)
			(1.500000000000000, 0.20994217175876)
			(1.510000000000000, 0.20554261240143)
			(1.520000000000000, 0.20115257839999)
			(1.530000000000000, 0.19677142867160)
			(1.540000000000000, 0.19240038760090)
			(1.550000000000000, 0.18803805969437)
			(1.560000000000000, 0.18368701672100)
			(1.570000000000001, 0.17934053320628)
			(1.580000000000001, 0.17500495807323)
			(1.590000000000001, 0.17067814782447)
			(1.600000000000001, 0.16635987862178)
			(1.610000000000001, 0.16205039936975)
			(1.620000000000001, 0.15774807112006)
			(1.630000000000001, 0.15345663048244)
			(1.640000000000001, 0.14917076022249)
			(1.650000000000001, 0.14489408207824)
			(1.660000000000001, 0.14062562281695)
			(1.670000000000001, 0.13636604965566)
			(1.680000000000001, 0.13211218764025)
			(1.690000000000001, 0.12786737822065)
			(1.700000000000001, 0.12363034876054)
			(1.710000000000001, 0.11940106890715)
			(1.720000000000001, 0.11518022748817)
			(1.730000000000001, 0.11096551500483)
			(1.740000000000001, 0.10675919856466)
			(1.750000000000001, 0.10256038953321)
			(1.760000000000001, 0.09836906292868)
			(1.770000000000001, 0.09418514862290)
			(1.780000000000001, 0.09000874127443)
			(1.790000000000001, 0.08583974650788)
			(1.800000000000001, 0.08167827444427)
			(1.810000000000001, 0.07752391068685)
			(1.820000000000001, 0.07337691459083)
			(1.830000000000001, 0.06923739804141)
			(1.840000000000001, 0.06510527011202)
			(1.850000000000001, 0.06098042637290)
			(1.860000000000001, 0.05686292279838)
			(1.870000000000001, 0.05275294155011)
			(1.880000000000001, 0.04864995989894)
			(1.890000000000001, 0.04455458629518)
			(1.900000000000001, 0.04046656455712)
			(1.910000000000001, 0.03638589401328)
			(1.920000000000001, 0.03231269399742)
			(1.930000000000001, 0.02824697490587)
			(1.940000000000001, 0.02418878626434)
			(1.950000000000001, 0.02013816535973)
			(1.960000000000001, 0.01609517998279)
			(1.970000000000001, 0.01205976089483)
			(1.980000000000001, 0.00803207891577)
			(1.990000000000001, 0.00401216992413)
			(2.000000000000001, -0.00000000000000)
			(2.010000000000001, -0.00400423438196)
			(2.020000000000000, -0.00800047582151)
			(2.030000000000000, -0.01198861512476)
			(2.040000000000000, -0.01596866249842)
			(2.050000000000000, -0.01994047445045)
			(2.060000000000000, -0.02390394736423)
			(2.069999999999999, -0.02785902050266)
			(2.079999999999999, -0.03180561528551)
			(2.089999999999999, -0.03574354795885)
			(2.099999999999999, -0.03967304905697)
			(2.109999999999999, -0.04359341471241)
			(2.119999999999998, -0.04750479376030)
			(2.129999999999998, -0.05140693724075)
			(2.139999999999998, -0.05529995887472)
			(2.149999999999998, -0.05918358842042)
			(2.159999999999997, -0.06305761878540)
			(2.169999999999997, -0.06692184218363)
			(2.179999999999997, -0.07077636335484)
			(2.189999999999997, -0.07462134478120)
			(2.199999999999997, -0.07845567954016)
			(2.209999999999996, -0.08227922542959)
			(2.219999999999996, -0.08609273407798)
			(2.229999999999996, -0.08989542553718)
			(2.239999999999996, -0.09368724733856)
			(2.249999999999996, -0.09746798677154)
			(2.259999999999995, -0.10123740592842)
			(2.269999999999995, -0.10499537246051)
			(2.279999999999995, -0.10874166519422)
			(2.289999999999995, -0.11247586883981)
			(2.299999999999994, -0.11619787095701)
			(2.309999999999994, -0.11990731490311)
			(2.319999999999994, -0.12360418503588)
			(2.329999999999994, -0.12728822018976)
			(2.339999999999994, -0.13095894969865)
			(2.349999999999993, -0.13461748277327)
			(2.359999999999993, -0.13826013029022)
			(2.369999999999993, -0.14188979674361)
			(2.379999999999993, -0.14550499662995)
			(2.389999999999993, -0.14910579913493)
			(2.399999999999992, -0.15269179106263)
			(2.409999999999992, -0.15626261700118)
			(2.419999999999992, -0.15981790980603)
			(2.429999999999992, -0.16335745672660)
			(2.439999999999992, -0.16688087255307)
			(2.449999999999991, -0.17038807860240)
			(2.459999999999991, -0.17387825934420)
			(2.469999999999991, -0.17735102827419)
			(2.479999999999991, -0.18080649157755)
			(2.489999999999990, -0.18424409536444)
			(2.499999999999990, -0.18766335329301)
			(2.509999999999990, -0.19106410568394)
			(2.519999999999990, -0.19444576642479)
			(2.529999999999990, -0.19780795300380)
			(2.539999999999989, -0.20115035075907)
			(2.549999999999989, -0.20447251664204)
			(2.559999999999989, -0.20777394784904)
			(2.569999999999989, -0.21105428101048)
			(2.579999999999989, -0.21431305568764)
			(2.589999999999988, -0.21754982619781)
			(2.599999999999988, -0.22076410392992)
			(2.609999999999988, -0.22395543078355)
			(2.619999999999988, -0.22712336399876)
			(2.629999999999987, -0.23026739842997)
			(2.639999999999987, -0.23338740995555)
			(2.649999999999987, -0.23648215849377)
			(2.659999999999987, -0.23955132960385)
			(2.669999999999987, -0.24259389718890)
			(2.679999999999986, -0.24561056132103)
			(2.689999999999986, -0.24860021011900)
			(2.699999999999986, -0.25156212212344)
			(2.709999999999986, -0.25449591119472)
			(2.719999999999986, -0.25740111350287)
			(2.729999999999985, -0.26027666269948)
			(2.739999999999985, -0.26312221450411)
			(2.749999999999985, -0.26593720706011)
			(2.759999999999985, -0.26872084117321)
			(2.769999999999984, -0.27147273817115)
			(2.779999999999984, -0.27419210200644)
			(2.789999999999984, -0.27687831381503)
			(2.799999999999984, -0.27953067231286)
			(2.809999999999984, -0.28214847635276)
			(2.819999999999983, -0.28473106870955)
			(2.829999999999983, -0.28727770961519)
			(2.839999999999983, -0.28978768599583)
			(2.849999999999983, -0.29226029608319)
			(2.859999999999983, -0.29469473605976)
			(2.869999999999982, -0.29709037078252)
			(2.879999999999982, -0.29944620834474)
			(2.889999999999982, -0.30176156977674)
			(2.899999999999982, -0.30403562205055)
			(2.909999999999981, -0.30626752784283)
			(2.919999999999981, -0.30845662259775)
			(2.929999999999981, -0.31060207479888)
			(2.939999999999981, -0.31270279996068)
			(2.949999999999981, -0.31475800547092)
			(2.959999999999980, -0.31676687428925)
			(2.969999999999980, -0.31872866635086)
			(2.979999999999980, -0.32064208247602)
			(2.989999999999980, -0.32250630830929)
			(2.999999999999980, -0.32432043956026)
			(3.009999999999979, -0.32608367968838)
			(3.019999999999979, -0.32779504152692)
			(3.029999999999979, -0.32945328878806)
			(3.039999999999979, -0.33105752520099)
			(3.049999999999979, -0.33260655555155)
			(3.059999999999978, -0.33409971587601)
			(3.069999999999978, -0.33553574728986)
			(3.079999999999978, -0.33691357846186)
			(3.089999999999978, -0.33823211037759)
			(3.099999999999977, -0.33949029895932)
			(3.109999999999977, -0.34068690800246)
			(3.119999999999977, -0.34182083890799)
			(3.129999999999977, -0.34289093514060)
			(3.139999999999977, -0.34389609606682)
			(3.149999999999976, -0.34483472015304)
			(3.159999999999976, -0.34570616221129)
			(3.169999999999976, -0.34650886555618)
			(3.179999999999976, -0.34724165832037)
			(3.189999999999976, -0.34790326219103)
			(3.199999999999975, -0.34849232653186)
			(3.209999999999975, -0.34900756951943)
			(3.219999999999975, -0.34944763537125)
			(3.229999999999975, -0.34981117211113)
			(3.239999999999974, -0.35009692289500)
			(3.249999999999974, -0.35030324476014)
			(3.259999999999974, -0.35042884601828)
			(3.269999999999974, -0.35047235281975)
			(3.279999999999974, -0.35043210831200)
			(3.289999999999973, -0.35030667478891)
			(3.299999999999973, -0.35009471130561)
			(3.309999999999973, -0.34979461053989)
			(3.319999999999973, -0.34940491643561)
			(3.329999999999973, -0.34892374094380)
			(3.339999999999972, -0.34834966914931)
			(3.349999999999972, -0.34768106913197)
			(3.359999999999972, -0.34691629231509)
			(3.369999999999972, -0.34605352038944)
			(3.379999999999971, -0.34509132846215)
			(3.389999999999971, -0.34402783553170)
			(3.399999999999971, -0.34286139031357)
			(3.409999999999971, -0.34159005284098)
			(3.419999999999971, -0.34021210049388)
			(3.429999999999970, -0.33872569604677)
			(3.439999999999970, -0.33712898326448)
			(3.449999999999970, -0.33542008884632)
			(3.459999999999970, -0.33359708681605)
			(3.469999999999970, -0.33165804555136)
			(3.479999999999969, -0.32960098314860)
			(3.489999999999969, -0.32742392210608)
			(3.499999999999969, -0.32512481885678)
			(3.509999999999969, -0.32270155394265)
			(3.519999999999968, -0.32015219296789)
			(3.529999999999968, -0.31747454530427)
			(3.539999999999968, -0.31466645539601)
			(3.549999999999968, -0.31172575421456)
			(3.559999999999968, -0.30865022787453)
			(3.569999999999967, -0.30543764196157)
			(3.579999999999967, -0.30208573425395)
			(3.589999999999967, -0.29859218098896)
			(3.599999999999967, -0.29495464595936)
			(3.609999999999967, -0.29117075430692)
			(3.619999999999966, -0.28723812365275)
			(3.629999999999966, -0.28315430044782)
			(3.639999999999966, -0.27891681791911)
			(3.649999999999966, -0.27452316226305)
			(3.659999999999966, -0.26997082379019)
			(3.669999999999965, -0.26525715703262)
			(3.679999999999965, -0.26037957940691)
			(3.689999999999965, -0.25533543183205)
			(3.699999999999965, -0.25012202585055)
			(3.709999999999964, -0.24473663003870)
			(3.719999999999964, -0.23917648463829)
			(3.729999999999964, -0.23343878110644)
			(3.739999999999964, -0.22752067879812)
			(3.749999999999964, -0.22141928035245)
			(3.759999999999963, -0.21513167182581)
			(3.769999999999963, -0.20865537575585)
			(3.779999999999963, -0.20198569773738)
			(3.789999999999963, -0.19512149241866)
			(3.799999999999963, -0.18805896893099)
			(3.809999999999962, -0.18079497695720)
			(3.819999999999962, -0.17332651081915)
			(3.829999999999962, -0.16565001182239)
			(3.839999999999962, -0.15776237795998)
			(3.849999999999961, -0.14966030216862)
			(3.859999999999961, -0.14134047221808)
			(3.869999999999961, -0.13279936435550)
			(3.879999999999961, -0.12403365373261)
			(3.889999999999961, -0.11503978145547)
			(3.899999999999960, -0.10581421586687)
			(3.909999999999960, -0.09635336029964)
			(3.919999999999960, -0.08665357320237)
			(3.929999999999960, -0.07671116458259)
			(3.939999999999960, -0.06652238713814)
			(3.949999999999959, -0.05608340357927)
			(3.959999999999959, -0.04539052629184)
			(3.969999999999959, -0.03443973001283)
			(3.979999999999959, -0.02322703466383)
			(3.989999999999958, -0.01174847758047)
			(3.999999999999958, -0.00000000000005)
			(4,0)
			};
		\end{axis}
	\end{tikzpicture}
	\caption{The Fourier coefficient \eqref{eq:Fcoeff} for $1 \leq p \leq 4$.}
	\label{fig:plot}
\end{figure}

\begin{proof}
	We begin with the case $p=\infty$, where we need to establish the estimate
	\[\left\|z_1^3+z_2^3+z_1z_2z_3-\varepsilon z_3^3\right\|_{H^\infty(\mathbb{T}^3)} < 3\]
	to see that $\varphi$ is not a Hilbert point in $H^\infty(\mathbb{T}^3)$ by \eqref{eq:cpoint}. To this end set $\zeta_1:=\overline{z_1}z_2$ and $\zeta_2:=\overline{z_1}^2 z_2 z_3 $ so that our task is to show that
	\[\left\| 1+\zeta_1+\zeta_2-\varepsilon \overline{\zeta_1}\zeta_2^3\right\|_{H^\infty(\mathbb{T}^3)} < 3. \]
	Now if $|1+\zeta_1|\leq 2-2\varepsilon$ or $|1+\zeta_2| \leq 2-2\varepsilon$, then trivially $\left|1+\zeta_1+\zeta_2-\varepsilon \overline{\zeta_1} \zeta_2^3 \right| \leq 3-\varepsilon$. It therefore suffices to consider $\zeta_1$ and $\zeta_2$ such that $|1+\zeta_1| > 2-2\varepsilon$ or $|1+\zeta_2| > 2-2\varepsilon$. In this case, we may finish the proof by an easy computation which is essentially identical to that given in the proof of \cite[Lem.~2.5]{BOS21}.
	
	We assume from now on that $p < \infty$. It is clear that we may rewrite the Fourier series of $\varphi$ as a Fourier series in the variables $\zeta_1:=z_1^3$, $\zeta_2:=z_2^3$, $\zeta_3:=z_1 z_2 z_3$. Such a rewriting reveals, by symmetry, that the Fourier coefficients of $P(|\varphi|^{p-2} \varphi)$ with respect to the three monomials $z_1^3$, $z_2^3$, $z_1z_2z_3$ are identical. Since the function $\varphi$ is $3$-homogeneous, there can be at most one additional term in the Fourier series of $P(|\varphi|^{p-2} \varphi)$, namely a multiple of $z_3^3$. We deduce from this that $\varphi$ is a Hilbert point in $H^p(\mathbb{T}^3)$ if and only if $\Phi(p)=0$, where 
	\begin{equation}\label{eq:Fcoeff} 
		\Phi(p):=\int_{\mathbb{T}^3} |\varphi(z)|^{p-2} \varphi(z) \,\overline{z_3^3} \,dm_3(z). 
	\end{equation}
	Using the notation $\psi(\zeta):=\zeta_1+\zeta_2+\zeta_3$, we get by the change of variables introduced above that 
	\begin{equation}\label{eq:vanish} 
		\Phi(p) = \int_{\mathbb{T}^3} |\psi(\zeta)|^{p-2} \psi(\zeta) \, \zeta_1 \zeta_2 \overline{\zeta_3^3} \,dm_3(\zeta). 
	\end{equation}
	Assume that $p-2=2n$, where $n$ is a nonnegative integer and expand
	\[\big(\psi(\zeta)\big)^{n+1} = \sum_{|\alpha|=n+1} \binom{n+1}{\alpha} \zeta^\alpha \qquad \text{and}\qquad \big(\overline{\psi(\zeta)}\big)^n = \sum_{|\beta|=n} \binom{n}{\beta} \overline{\zeta^{\beta}}.\]
	We only get a contribution to \eqref{eq:vanish} for $\alpha=\beta+(-1,-1,3)$, when
	\[\binom{n+1}{\alpha} \binom{n}{\beta} = (n+1)\binom{n}{\beta}^2 \frac{\beta_1 \beta_2}{(\beta_3+1)(\beta_3+2)(\beta_3+3)}.\]
	This shows that 
	\begin{equation}\label{eq:Phin1} 
		\Phi(2(n+1)) = (n+1) \sum_{|\beta|=n} \binom{n}{\beta}^2 \frac{\beta_1 \beta_2}{(\beta_3+1)(\beta_3+2)(\beta_3+3)}. 
	\end{equation}
	Note that the numerator $\beta_1 \beta_2$ ensures that $\Phi(2)=\Phi(4)=0$, so $\varphi$ is a Hilbert point in $H^2(\mathbb{T}^3)$ and $H^4(\mathbb{T}^3)$. It is also clear that $\Phi(2(n+1))>0$ for every integer $n\geq2$. We need an analytic expression for \eqref{eq:Phin1}. This can be established directly using Bergman norms when $n\geq2$. If we write $\psi_n(w)=(w_1+w_2+w_3)^n$, then 
	\begin{align*}
		\Phi(2(n+1)) &= \frac{(n+1)}{3!} \int_{\mathbb{D}^3} \left|\frac{\partial^2}{\partial w_1 \partial w_2} \psi_n(w)\right|^2\, 3(1-|w_3|^2)^2 \,dA_3(w) \\
		&= \frac{(n+1) n^2 (n-1)^2}{3!} \int_{\mathbb{D}^3} |\psi_{n-2}(w)|^2 \, 3(1-|w_3|^2)^2 \,dA_3(w), 
	\end{align*}
	where $dA_3(w):= dA(w_1)dA(w_2)dA(w_3)$. Returning to \eqref{eq:Fcoeff}, we have established the identity 
	\begin{equation}\label{eq:Phiident} 
		\Phi(p) = \binom{p/2}{3} \frac{(p-2)(p-4)}{4} \int_{\mathbb{D}^3} |w_1+w_2+w_3|^{p-6}\, 3(1-|w_3|^2)^2 \,dA_3(w) 
	\end{equation}
	when $p>4$ is a positive even integer. Since the sequence of positive integers violates the Blaschke condition in the right half-plane and since $\Phi$ grows at most exponentially as $p\to\infty$, it follows by analytic continuation that \eqref{eq:Phiident} is valid also for non-integer $p>4$. It breaks down at $p=4$ because integrability fails. The expression on the right-hand side of \eqref{eq:Phiident} is positive for $p>4$, so we get that the Fourier coefficient \eqref{eq:Fcoeff} does not vanish. Hence $\varphi$ is not a Hilbert point in $H^p(\mathbb{T}^3)$ for $p>4$. 
\end{proof}
\begin{remark}
	One could offer a rigorous computer assisted proof that $\Phi(p)\ne 0$ also when $1\leq p<4$, $p\neq 2$, by estimating the integral in \eqref{eq:vanish} using interval arithmetic in the intervals $[1,2)$ and $(2,4)$ and analyzing separately the behavior near $p = 2$ and $p = 4$. 
\end{remark}

We believe that the Fourier coefficients of 
\begin{equation}\label{eq:d3coeff} 
	|\zeta_1+\zeta_2+\zeta_3|^{p-2}(\zeta_1+\zeta_2+\zeta_3) = \sum_{|\alpha|=1} c_p(\alpha) \zeta^\alpha 
\end{equation}
may be of some independent interest. In the proof of Theorem~\ref{thm:orthocex} we investigated the Fourier coefficient corresponding to $\alpha=(-1,-1,3)$.

Our interest in \eqref{eq:d3coeff} stems from the fact that when $d=2$, we have easy access to all the corresponding Fourier coefficients. By a computation in \cite[Sec.~3]{Brevig19}, it follows that 
\begin{equation}\label{eq:ole} 
	|\zeta_1+\zeta_2|^{p-2} (\zeta_1+\zeta_2) = \sum_{|\alpha|=1} \frac{\Gamma(p)}{\Gamma(p/2+\alpha_1)\Gamma(p/2+\alpha_2)} \zeta^\alpha. 
\end{equation}
Let us sketch a different proof of \eqref{eq:ole} in the spirit of Theorem~\ref{thm:orthocex}. We consider first $p=2n$ for a positive integer $n$ and write 
\begin{equation}\label{eq:binomial} 
	|\zeta_1+\zeta_2|^{2n-2}(\zeta_1+\zeta_2) = \overline{(\zeta_1\zeta_2)^n} (\zeta_1+\zeta_2)^{2n-1} = \sum_{j=0}^{2n-1} \binom{2n-1}{j}z_1^{n-1-j}z_2^{j-n}, 
\end{equation}
to establish \eqref{eq:ole} when $p$ is an even integer. By analytic continuation as in the proof of Theorem~\ref{thm:orthocex}, we obtain \eqref{eq:d2rec} for $1 \leq p < \infty$. This proof shows that if $p=2n$, then the nonzero Fourier coefficients in \eqref{eq:ole} are precisely the entries in row $2n-1$ of Pascal's triangle. Since $|\zeta_1+\zeta_2|^2 = 2 + \zeta_1\overline{\zeta_2}+\overline{\zeta_1}\zeta_2$, we see that 
\begin{equation}\label{eq:d2rec} 
	c_{p+2}(\alpha_1,\alpha_2) = 2 c_p(\alpha_1,\alpha_2) + c_p(\alpha_1-1,\alpha_2+1) + c_p(\alpha_1+1,\alpha_2-1), 
\end{equation}
where $\alpha_1+\alpha_2=1$. If $p=2n$, then the recursion \eqref{eq:d2rec} corresponds to the three applications of Pascal's formula need to go from row $2n-1$ to row $2n+1$.

Returning to \eqref{eq:d3coeff}, we similarly expand $|\zeta_1+\zeta_2+\zeta_3|^2$ to get the recursion 
\begin{align*}
	c_{p+2}(\alpha_1,\alpha_2,\alpha_3) = 3c_p(\alpha_1,\alpha_2,\alpha_3) &+ c_p(\alpha_1+1,\alpha_2-1,\alpha_3) + c_p(\alpha_1+1,\alpha_2,\alpha_3-1) \\
	&+ c_p(\alpha_1-1,\alpha_2+1,\alpha_3) + c_p(\alpha_1,\alpha_2+1,\alpha_3-1) \\
	&+ c_p(\alpha_1-1,\alpha_2,\alpha_3+1) + c_p(\alpha_1,\alpha_2-1,\alpha_3+1) 
\end{align*}
where $\alpha_1+\alpha_2+\alpha_3=1$. Specializing to the case $p=2n$ for positive integers $n$ as above, we observe that the nonzero Fourier coefficients in \eqref{eq:d3coeff} correspond precisely to the entries in the slice $2n-1$ of the hexagonal Pascal's pyramid. We refer to the On-Line Encyclopedia of Integer Sequences \cite{OEIS} and note that our numbering of the slices differs by $1$. The numbers in the hexagonal Pascal's pyramid does not have a known closed form similar to the binomials appearing in \eqref{eq:binomial}, so it is not clear how to proceed to get a formula for general $1 \leq p < \infty$.

Further examples can be generated starting from any of the sets found in \cite[Sec.~3]{BOS21}. It is clear that for every $n>1$ we could construct functions that are Hilbert points in $H^p(\mathbb T^4)$ for $p=2, 4,\ldots , 2n$ and ``most likely'' for no other $p$ in the range $1\le p \le \infty$. We use quotation marks here to indicate that verifying rigorously the latter assertion would be difficult if not impossible.

Beyond inner functions, we have so far only seen polynomial Hilbert points in $H^p(\mathbb T^d)$ for $p\neq 2$, which reflects that our understanding of the general situation is very limited. We do not know, for instance, whether there exists an unbounded Hilbert point in $H^p(\mathbb T^d)$ for some $p\neq 2$. It remains also to be seen whether Hilbert points in $H^p(\mathbb T^d)$ may have an operator theoretic role to play when $d\ge 2$, as they 
do with such distinction when $d=1$, in view of Beurling's theorem.
\bibliographystyle{amsplain} 
\bibliography{cpoints}

\end{document}